\documentclass[12pt, reqno]{amsart}
\usepackage{amsfonts}
\usepackage{bbm}
\usepackage{amscd,amsfonts}
\usepackage{amssymb, eucal, amsfonts, amsmath, xypic, latexsym}
\usepackage{pifont}
\usepackage{mathrsfs,color}
\usepackage{amsthm,indentfirst,bm,fancyhdr,dsfont}
\usepackage{graphicx}
\usepackage[all]{xy}
\usepackage[CJKbookmarks=true]{hyperref}

\newtheorem{theorem}{Theorem}[section]
\newtheorem{lemma}[theorem]{Lemma}
\newtheorem{prop}[theorem]{Proposition}
\newtheorem{proposition}[theorem]{Proposition}

\theoremstyle{definition}

\newtheorem{defn}[theorem]{Definition}
\newtheorem{definition}[theorem]{Definition}
\newtheorem{example}[theorem]{Example}

\newtheorem{remark}[theorem]{Remark}

\newtheorem{question}[theorem]{Question}

\numberwithin{equation}{section}

\def\ggg{\mathfrak{g}}

\def\gl{\mathfrak{gl}}

\def\cb{\mathcal{B}}

\def\h{\mathfrak{h}}
\def\g{\mathfrak{g}}
\def\ggg{\mathfrak{g}}

\def\hhh{\mathfrak{h}}

\def\bbb{\mathfrak{b}}
\def\nnn{\mathfrak{n}}

\def\frakt{\mathfrak{T}}

\def\bbc{\mathbb{C}}
\def\bbf{\mathbb{F}}
\def\bbz{\mathbb{Z}}

\def\bbk{\mathbf{k}}
\def\bbn{\mathbb{N}}

\def\bo{{\bar 1}}
\def\bz{{\bar 0}}
\def\ev{{\text{ev}}}

\def\Lie{\text{Lie}}

\def\im{\text{im}}
\def\coker{\text{coker}}
\def\Ad{\text{Ad}}

\def\GL{\text{GL}}
\def\Hom{\text{Hom}}
\def\Aut{\text{Aut}}

\def\Nor{\text{Nor}}
\def\Dist{\text{Dist}}
\def\hmod{\text{-\bf{mod}}}
\def\id{\mathsf{id}}

\def\ttau{{}^\tau{\hskip-2pt}}

{\vskip-\lastskip\medskip
  \noindent
  {\em #1.}\enspace
  }%
{\qed\par\medskip
  }

\begin{document}

\title[Jantzen Filtration for modular super representaions]{Jantzen filtration and strong linkage principle for modular Lie superalgebras}
\author{Lei Pan and Bin Shu}

\address{Department of Mathematics, East China Normal University, Shanghai 200241, China}
\email{823172606@qq.com}

\address{Department of Mathematics, East China Normal University, Shanghai 200241, China}
\email{bshu@math.ecnu.edu.cn}

\thanks{
{{\it{Mathematics Subject Classification}} (2010):
Primary 17B50, Secondary 17B10, 17B20 and 17B45.
  {\it{The Key words}}: Basic classical Lie superalgebras, super Weyl groups, baby Verma modules, Jantzen filtration and sum formula, strong linkage principle.
This work is  supported partially by the NSFC (Grant No. 11671138), Shanghai Key Laboratory of PMMP (13dz2260400). }}

\begin{abstract}  In this paper, we introduce super Weyl groups, their distinguished elements and  properties for basic classical Lie superalgebras. Then we formulate Jantzen filtration for baby Verma modules in graded restricted module categories of basic classical Lie superalgebras over an algebraically closed field of odd characteristic, and prove a sum formula in the corresponding Grothendieck groups. We finally obtain a strong linkage principle in such categories.
\end{abstract}
\maketitle
\setcounter{tocdepth}{1}
\tableofcontents

\section{Introduction}

\subsection{} Recently, the study on representations of  basic classical Lie superalgebras and the corresponding algebraic supergroups in prime characteristic becomes an increasing interest (\textit{cf}. \cite{BruKl, BruKu,Ku, Marko, MarkoZ1, MarkoZ2, SW, WZ, WZ2,  ZS2, ZS3, Zhang, Z1, ZhengS}, {\sl{etc}}.),   in the meanwhile the study on the counterparts  over the complex number field is going on.
In such a topic, it is a key point to understand irreducible modules through standard modules (baby Verma modules). So far, very few information on this has been known. Recall in  the Bernstein-Gelfand-Gelfand categories of complex semisimple Lie algebras, and in  (rational) representation categories of modular reductive algebraic group and of their reductive Lie algebras associated with specific  tori of the corresponding algebraic groups, some kinds of artfully-constructed filtrations of standard modules introduced by Jantzen in \cite{J0} for complex semisimple Lie algebars, in \cite{J3, J2} for modular reductive algebraic groups and their Lie algebras,  is a very powerful tool to understand character formulas of irreducible modules through the ones of standard modules (\textit{cf}. \cite{Hum3, J0, J1, J2, J3}, {\sl{etc}}.).  In this paper we exploit the theory of Jantzen filtration for reductive Lie algebras in prime characteristic to the super case.

\subsection{} We will focus on restricted representations of a basic classical Lie superalgebra $\ggg$ over $\bbk$, an algebraically closed field of characteristic $p>2$. In order to establish a satisfactory theory on characters for restricted irreducible modules in modular cases, a usual way is to refine the restricted module category $U_0(\ggg)\hmod$ of $\ggg$ by introducing $\bbz$-graded structure.   So we will actually  work with a graded restricted module category $(U_0(\ggg), \frakt)\hmod$, via an additional $\frakt$-action for a maximal torus $\frakt$ of algebraic supergroup $G$ with $\Lie(G)=\ggg$ (see Definition \ref{UGTMOD}).

A super Weyl group $\hat W$ is introduced here, which is a subgroup of the symmetric group of the set of all Borel subalgebras, generated by all real reflections and odd reflections.
In our construction of Jantzen filtration for baby Verma modules in $(U_0(\ggg), \frakt)\hmod$,  an important ingredient is to find a distinguished element $\hat w_0$ of  the super Weyl group $\hat W$ for $\ggg$ (see Theorem \ref{longest element thm}), which plays a role as important as the longest elements of usual Weyl groups.

In order to  establish Jantzen filtration, apart from carrying some known information and methods by Jantzen in \cite{J2} forward to the Lie superalgebra from its even part,  we need do more in seeking for and  analysing  homomorphisms between (twisted) Verma modules, resulted from  odd roots. These (twisted) Verma modules are associated with twisted Borel subalgebras, corresponding different positive root sets arising from the reduced expression of $\hat w_0$. With the above, we finally obtain a satisfactory filtration and sum formula (see Theorem \ref{Main Thm}). As a consequence, we can read off a strong linkage principle from the sum formula (see Theorem \ref{stronglinkage2} and  Theorem \ref{linkage thm}).


\subsection{} One can compare some related results over complex numbers (\textit{cf}. \cite[Chapter 10]{Mu1}, \cite{Mu2} and \cite{SuZh}).

\section{Preliminaries}

Throughout the article, the notations of vector spaces (\textit{resp.} modules, subalgebras) means vector superspaces (\textit{resp.} super modules, super subalgebras). For simplicity, we omit the adjunct word "super". For a superspace $V=V_\bz+V_\bo$, the homogeneous element $v$ is  assumed to have a parity $|v|\in\{\bo,\bz\}$. All vector spaces are defined over $\bbk$ for an algebraically closed field of characteristic $p>2$.

\subsection{Lie superalgebras and algebraic supergroups}

In this section, we will recall some knowledge on basic classical Lie superalgebras along with the corresponding algebraic supergroups. We refer the readers to  \cite{CW, K, Mu1} for Lie superalgebras, \cite{BruKl, FG, M, SW} for algebraic supergroups. 

\subsubsection{Basic classical Lie superalgebras} \label{basicLiesuper}
Following \cite[\S1]{CW}, \cite[\S2.3-\S2.4]{K}, \cite[\S1]{K2} and \cite[\S2]{WZ},  we recall the list of basic classical Lie superalgebras over $\bbf$ for $\bbf=\bbc$ or $\bbf=\bbk$ where $\bbk$ is an algebraically closed field of odd characteristic $p$.
These Lie superalgebras, with even parts being Lie algebras of reductive algebraic groups, are simple over $\bbf$ (the
general linear Lie superalgebras, though not simple, are also included), and they admit an even non-degenerate supersymmetric invariant bilinear form in the following sense.
\begin{defn}\label{form}
Let $V=V_{\bar0}\oplus V_{\bar1}$ be a $\mathbb{Z}_2$-graded space and $(\cdot,\cdot)$ be a bilinear form on $V$.
\begin{itemize}
\item[(1)] If $(a,b)=0$ for any $a\in V_{\bar0}, b\in V_{\bar1}$, then $(\cdot,\cdot)$ is called even;
\item[(2)] if $(a,b)=(-1)^{|a||b|}(b,a)$ for any homogeneous elements $a,b\in V$, then $(\cdot,\cdot)$ is called supersymmetric;
\item[(3)] if $([a,b],c)=(a,[b,c])$ for any homogeneous elements $a,b,c\in V$, then $(\cdot,\cdot)$ is called invariant;
\item[(4)] if one can conclude from $(a,V)=0$ that $a=0$, then $(\cdot,\cdot)$ is called non-degenerate.
\end{itemize}
\end{defn}

Note that when $\bbf=\bbk$ is a field of characteristic $p>2$, there are restrictions on $p$; for example, as shown in \cite[Table 1]{WZ}. So we have the following list

\vskip0.3cm
\begin{center}\label{Table 1}

({\sl{Table 1}}):\label{Table} basic classical Lie superalgebras over $\bbk$
\vskip0.3cm
\begin{tabular}{ccc}
\hline
 $\frak{g}_\bbk$ & $\ggg_{\bar 0}$  & the restriction of $p$ when $\bbf=\bbk$\\
\hline
$\frak{gl}(m|n$) &  $\frak{gl}(m)\oplus \frak{gl}(n)$                &$p>2$            \\
$\frak{sl}(m|n)$ &  $\frak{sl}(m)\oplus \frak{sl}(n)\oplus \bbk$    & $p>2, p\nmid (m-n)$   \\
$\frak{osp}(m|n)$ & $\frak{so}(m)\oplus \frak{sp}(n)$                  & $p>2$ \\
$\text{D}(2,1,\bar a)$   & $\frak{sl}(2)\oplus \frak{sl}(2)\oplus \frak{sl}(2)$        & $p>3$   \\
$\text{F}(4)$            & $\frak{sl}(2)\oplus \frak{so}(7)$                  & $p>15$  \\
$\text{G}(3)$            & $\frak{sl}(2)\oplus \text{G}_2$                    & $p>15$     \\
\hline
\end{tabular}

\end{center}

\vskip0.3cm


\subsubsection{Algebraic supergroups and restricted Lie superalgebras}\label{2.1}
For a given basic Lie superalgebra listed in \S\ref{basicLiesuper}, there is an algebraic supergroup $G_\bbf$  satisfying $\Lie(G_\bbf)=\ggg_\bbf$ such that
\begin{itemize}
\item[(1)] $G_\bbf$ has a subgroup scheme $(G_\bbf)_\ev$ which is an ordinary connected reductive group with $\Lie((G_\bbf)_\ev)=(\ggg_\bbf)_\bz$;
\item[(2)] there is a well-defined action of $(G_\bbf)_\ev$ on $\ggg_\bbf$, reducing to the adjoint action of $(\ggg_\bbf)_\bz$.
    \end{itemize}
The above algebraic supergroup can be constructed in the Chevalley group way, which we call an algebraic supergroup of Chevalley type (or,  Chevalley supergroup in \cite{FG}). The pair ($(G_\bbf)_\ev, \ggg_\bbf)$ constructed in this way is called a Chevalley super Harish-Chandra pair (\textit{cf}.  \cite[\S5]{FG}, \cite[\S3.3]{FG2}, \cite{Kostant} and \cite{Ko}).
Partial results on $G_\bbf$ and $(G_\bbf)_\ev$ can be found in \cite[Part II, \S2.2]{Ber}, \cite{FG}, \cite{FG2}, \cite{Gav}, \cite{Ko}, {\sl etc.}. In the present paper, we will call $(G_\bbf)_\ev$ the purely even subgroup of $G_\bbf$. When the ground field $\bbf=\bbk$ is of odd prime characteristic $p$, one easily knows that $\ggg_\bbk$ is a restricted Lie superalgebra (\textit{cf}. \cite[Definition 2.1]{SW} and \cite{SZ}) in the following sense.

\begin{defn}\label{restricted}
A Lie superalgebra ${\ggg}_{\bbk}=({\ggg}_{\bbk})_{\bar{0}}\oplus({\ggg}_{\bbk})_{\bar{1}}$ over ${\bbk}$ is called a restricted Lie superalgebra,
if there is a $p$-th power map $({\ggg}_{\bbk})_{\bar{0}}\rightarrow({\ggg}_{\bbk})_{\bar{0}}$, denoted as $(-)^{[p]}$, satisfying

(a) $(k x)^{[p]}=k^p x^{[p]}$ for all $k\in{\bbk}$ and $ x\in({\ggg}_{\bbk})_{\bar{0}}$;

(b) $[x^{[p]}, y]=(\text{ad}x)^p(y)$ for all $x\in({\ggg}_{\bbk})_{\bar{0}}$ and $y\in{\ggg}_{\bbk}$;

(c) $(x+y)^{[p]}=x^{[p]}+ y^{[p]}+\sum\limits_{i=1}^{p-1}s_i(x, y)$ for all $x, y\in({\ggg}_{\bbk})_{\bar{0}}$, where $is_i(x, y)$ is the
coefficient of $\lambda^{i-1}$ in $(\text{ad}(\lambda x+ y))^{p-1}(x)$.
\end{defn}

Let $\ggg_\bbk$ be a restricted Lie superalgebra.  For each $x\in (\ggg_\bbk)_\bz$, the element $x^p-x^{[p]}\in U(\ggg_\bbk)$ is central by Definition \ref{restricted}, and all of which generate a central subalgebra of $U(\ggg_\bbk)$.
Let $\{w_1,\cdots,w_c\}$ and $\{w'_1,\cdots,w'_d\}$ be the basis of $(\ggg_\bbk)_\bz$ and $(\ggg_\bbk)_\bo$ respectively.
For a given $\chi\in (\ggg_\bbk)_\bz^*$,  let $J_\chi$ be the ideal of the universal enveloping algebra $U(\ggg_\bbk)$ generated by the even central elements $\bar w^p-\bar w^{[p]}-\chi(\bar w)^p$ for all $\bar w\in (\ggg_\bbk)_\bz$. The quotient algebra $U_\chi(\ggg_\bbk) := U(\ggg_\bbk)\slash J_\chi$ is called the reduced enveloping algebra with $p$-character $\chi$. We often regard $\chi\in \ggg_\bbk^*$ by letting $\chi((\ggg_\bbk)_\bo) = 0$. If $\chi= 0$, then $U_0(\ggg_\bbk)$ is called
the restricted enveloping algebra. It is a direct consequence from the PBW theorem that the $\bbk$-algebra $U_\chi(\ggg_\bbk)$ is of dimension $p^c 2^d$, and has a basis
$$\{w_1^{a_1}\cdots w_c^{a_c}(w'_1)^{b_1}\cdots(w'_d)^{b_d}
\mid 0\leqslant a_i<p, b_j \in\{0, 1\} \mbox{ for  all }1\leqslant i\leqslant c, 1\leqslant j\leqslant d\}.$$

 A supermodule $(\rho,V)$ for a restricted Lie superalgebra
$(\g_\bbk,[p])$ is said to be {\it{restricted}} if $\rho$ satisfies for all
$x\in (\g_\bbk)_{\bar 0}$
\begin{align} \rho(x)^p-\rho(x^{[p]})=0.   \label{restrictedmod} \end{align}
All restricted modules of $\g_\bbk$ constitute a full subcategory of the
$\g_\bbk$-module category, which coincides with the $U_0(\g_\bbk)$-module
category, denoted by $U_0(\g_\bbk)\hmod$.

\subsubsection{Root structural information} \label{rootsubsub}
From now on, we will always assume, throughout the remaining of the paper,  that {\sl  all base fields are $\bbk$, which is algebraically closed and of odd characteristic $p$ satisfying the condition listed in (Table 1)}. The notations involving the base field will be simplified. For example, an algebraic supergroups $G_\bbk$ and its Lie algebra $\ggg_\bbk$ will be directly denoted by $G$ and $\ggg$, with the subscript $\bbk$ being omitted.

For a given basic Lie superalgebra $\ggg$, by the arguments as above $\ggg$ is a Lie superalgebra of an algebraic supergroup $G$ whose purely even group $G_\ev$ is reductive. Fix a standard maximal torus $\frakt$ of $G_\ev$ which consists of diagonal matrices in $G_\ev$, and set $\hhh=\Lie(\frakt)$.  Then $\mathfrak{h}$ is a Cartan subalgebra of ${\ggg}$. Let $\Delta$ be a root system of ${\ggg}$ relative to $\mathfrak{h}$ whose simple root system $\Pi=\{\alpha_1,\cdots,\alpha_l\}$ is standard, corresponding to the standard diagrams in the sense of \cite[\S1.3]{CW} (\textit{cf}. \cite[Proposition 1.5]{K2}). Let $\Delta^+$ be the corresponding positive system in $\Delta$, and put $\Delta^-:=-\Delta^+$. Let ${\ggg}=\mathfrak{n}^-\oplus\mathfrak{h}\oplus\mathfrak{n}^+$ be the corresponding triangular decomposition of ${\ggg}$. There is a canonical Borel subalgebra $\bbb:=\hhh+\nnn^+$. Furthermore, $\ggg=\ggg_{\bz}+\ggg_{\bo}$, and $\ggg_{\bz}=\hhh+\sum_{\alpha\in \Delta_0}\ggg_\alpha$ and $\ggg_\bo=\sum_{\beta\in \Delta_1}\ggg_\beta$. The set of even roots $\Delta_0$ is occasionally  denoted by $\Delta_\ev$, and the set of odd roots $\Delta_1$  by $\Delta_{\text{odd}}$ for clarification.

As discussed in \S\ref{basicLiesuper}, there exists a non-degenerate even invariant super-symmetric bilinear form $(\cdot,\cdot)$ on $\ggg$, which restricts to a non-degenerate form $(\cdot,\cdot)$ on $\hhh$ and on $\hhh^*$.  Especially, $(\cdot,\cdot)$ can be defined in $\Delta$.  A root $\gamma\in \Delta$ is called isotropic if $(\alpha,\alpha)=0$ (note an isotropic root is necessarily an odd root). Following \cite{Mu1}£¬ we set

$$\overline \Delta_0:=\{\alpha\in \Delta_0\mid {\alpha\over 2}\notin \Delta_1\},\;\; \overline\Delta_1:=\{\beta\in\Delta_1\mid 2\beta\notin\Delta_0\}. $$

By Lemma 8.3.2 of \cite{Mu1}, $\overline\Delta_1$ is just the set of isotropic roots. And then the set of nonisotropic roots is
\begin{align} \label{nonisotropic}
\Delta_{\text{nonisotropic}}=\Delta_0\cup(\Delta_1\backslash \overline\Delta_1).
\end{align}




\subsection{Chevalley bases and coroots} \label{Chevalley Basis}

 Under some mild condition (the one listed in (Table 1) is sufficient),  we can write $\ggg=\ggg_\bbz\otimes_\bbz \bbk$ where $\ggg_\bbz$ is a $\bbz$-form of $\ggg$,   and $\ggg_\bbz=\hhh_\bbz+\sum_{\theta\in \Delta}(\ggg_\bbz)_\theta$  has a Chevalley basis $\{X_\alpha\in (\ggg_\bbz)_\alpha, Y_\alpha\in(\ggg_\bbz)_{-\alpha}\mid \alpha\in \Delta^+\}\cup \{H_i\in\hhh_\bbz\mid i=1,\cdots, r:=\dim\hhh\}$ satisfying the Chevalley basis axioms as presented as in \cite[\S3.2]{FG}
 (see \cite[Theorem 3.3.1]{FG}, or \cite{IK} \cite{SW}). Especially, $\ggg_\bbz=\hhh_\bbz+\sum_{\theta\in \Delta}(\ggg_\bbz)_\theta$
 such that $\{H_i\mid i=1,\cdots,r\}$ is a $\bbz$-basis of $\hhh_\bbz$, and $(\ggg_\bbz)_\alpha=\bbz X_\alpha$ and $(\ggg_\bbz)_{-\alpha}=\bbz Y_\alpha$ for $\alpha\in \Delta^+$ with $[X_\alpha,Y_\alpha]=H_\alpha$. Furthermore, associated with each root $\alpha\in \Delta^+$,  there is a vector $H_\alpha\in \hhh_\bbz$ satisfying $\hhh_\bbz=\bbz\mbox{-span}\{H_\alpha\mid \alpha\in \Delta^+\}$ and
 $(\alpha,\beta)=(H_\alpha,H_\beta)\in \bbz$ for $\alpha,\beta\in \Delta^+$, and $\alpha(H_\alpha)=2$ for any $\alpha\in \Delta^+\backslash \overline\Delta_1^+$ (\textit{cf}.  \cite[\S3.1, \S3.2]{FG}).

 Recall that the character group $X(\frakt)$ of $\frakt$
is a free abelian group of rank $r:=\dim \frakt$, identified with $\bbz^r$.
 Associated with $\alpha\in \Delta^+$, by the construction of Chevalley supergroups we have  a  one-parameter  multiplicative supersubgroup $\chi_\alpha$ such that $\chi_\alpha(t)\in \frakt$ for $t\in \bbk^{\times}:=\bbk \backslash \{0\}$ and  $\lambda(\chi_\alpha(t))=t^{\lambda(H_\alpha)}$ for $\lambda\in X(\frakt)$ (\textit{cf}. \cite[\S5.2]{FG}). By the same way as  the theory of algebraic groups (\cite[\S16.1]{Hum2} or \cite[\S{II.} 1.3]{J1}), we can assign the set of pairs of $\lambda, \chi_\alpha$ to $\bbz$  via $\langle\lambda,\chi_\alpha\rangle=\lambda(H_\alpha)$. Usually, $\chi_\alpha$ is called a coroot corresponding to $\alpha$ (\textit{cf}. \cite[\S{II.} 1.3]{J1}). One can describe precisely $X(\frakt)$ for basic classical cases  (see \cite{BruKu} and \cite{SW} for type $A, B, C$ and $D$). The following
  general fact is clear. 

\begin{lemma}\label{character gp} $\bbz\Pi\subset X(\frakt)$.
\end{lemma}

\begin{proof} It follows from \cite[\S5.2]{FG}.
\end{proof}

 \subsection{} \label{tau} There is a standard involution $\tau_0\in \Aut(\ggg)$ which interchanges $X_\alpha$ with $Y_\alpha$ for all $\Phi^+$ and stabilizes $\hhh$ with $\tau|_\hhh=-\id_\hhh$. Actually, this $\tau_0$ is the differential of an automorphism $\tau\in \Aut(G)$ which satisfies $\tau(t)=t^{-1}$ for all $t\in \frakt$. This $\tau_0$ can be realized as $\tau_0(X)=-X^\mathsf{st}$ for  $\ggg=\gl(m|n)$ where $\mathsf{st}$ means super transpose (see \cite[5.6.9, page 128]{Mu1}), and $\tau$ can be realized as $\tau(x)=(x^{-1})^\mathsf{st}$ for $G=\GL(m|n)$ and $x\in G(A)$ for any super commutative $\bbk$-algebra $A$. Obviously, $\tau$ induces $-\id$ on $X(\frakt)$.

\section{Super Weyl groups and their distinguished elements}

Keep the notations and assumptions as in \S\ref{rootsubsub}. In particular, let $\ggg$ be an any given basic Lie superalgebra over $\bbk$, with $\Lie(G)=\ggg$ and Chevalley super Harish-Chandra pair $(G_\ev,\ggg)$, and $\frakt$ be a fixed maximal torus $G_\ev$ which consists of diagonal matrices in $G_\ev$, and $\hhh=\Lie(\frakt)$.

\subsection{Odd  and real reflections} Beside real reflections associated to even roots as for semisimple Lie algebras,  one can define odd reflections associated to isotropic odd simple roots. Both real and odd reflections permute the fundamental systems of a root system.  There is a fundamental fact on an odd simple root as following.

\begin{lemma} (\cite[Lemma 1.30]{CW}) \label{oddreflem} Let $\ggg$ be a basic Lie superalgebra and maintain the notations as in \S\ref{rootsubsub}. For a given isotropic simple root $\gamma$  in  $\Pi$, set  $\Delta^+_\gamma:=\{-\gamma\}\cup \Delta^+\backslash\{\gamma\}$. Then $\Delta^+_\gamma$ is a new positive root system whose corresponding fundamental root system $\Pi_\gamma$ is given by
$$ \Pi_\gamma=\{\alpha\in\Pi\mid (\alpha,\gamma)=0,\alpha\ne\gamma\}\cup\{\alpha+\gamma\mid\alpha\in\Pi, (\alpha,\gamma)\ne 0\}\cup\{-\gamma\}.$$
\end{lemma}

For an isotropic simple root $\gamma$ with respect to the simple root system $\Pi(\frak{B})$ of a given Borel subalgebra $\frak{B}$ (and the corresponding positive system $\Delta(\frak{B})^+$), we can write $\frak{B}=\hhh+\sum_{\theta \in\Delta(\frak{B})^+}\ggg_\theta$. The above lemma is actually to give an operator which is to transform $\frak{B}$ to a new Borel subalgebra
$$\frak{B}^\gamma:=\hhh+\sum_{\theta \in\Delta(\frak{B}^\gamma)^+}\ggg_\theta$$
 with $\Delta(\frak{B}^\gamma)^+=\Delta(\frak{B})^+_\gamma$ and $\Pi(\frak{B}^\gamma)=\Pi(\frak{B})_\gamma$. Such operation is called odd reflection (with respect to $\gamma$), which is denoted by $r_\gamma$. Obviously, $r_{-\gamma}(\frak{B}^\gamma)=\frak{B}$. We will identify $r_\gamma$ with $r_{-\gamma}$ when we deal with such a pair $\{\frak{B}, \frak{B}^\gamma\}$, which will be helpful to introduce so-called super Weyl groups (see Definition \ref{superWeyl}), not making any confusion.

For a non-isotropic root,  by (\ref{nonisotropic}) we know that it is either an even root $\alpha$ , or an odd root $\beta$ with $2\beta\in \Delta_0$. So for a given non-isotropic root $\theta$ we can set $s_\theta(\lambda)=\lambda-2(\lambda,\theta)\slash (\theta,\theta)\theta$, which becomes a linear map on $\hhh^*$.  Note that $s_\beta=s_{2\beta}$ for any $\beta\in \Delta_1\backslash \overline\Delta_1$. All of them are called real reflections. The Weyl group $W$ for $\ggg_\bz$ is just generated by those real reflections $s_\theta$ for all $\theta\in \overline\Delta_0^+\cup(\Delta_1^+\backslash\overline\Delta_1^+)$ (note that $s_\theta=s_{-\theta}$), as a subgroup of $\GL(\hhh^*)$,  which stabilize $\Delta$, $\Delta_0$ and $\Delta_1$ respectively.

In the sequent arguments, we will unify the notations of odd and real reflections, by using $\hat r_\theta$ respect to the root $\theta$.  We first have an  observation  that  if $\theta\in\Pi$ then $\hat{r}_{\theta}\Delta^+$ and $\Delta^+$ have differences by at most two roots, i.e. $\#(\hat{r}_{\theta}\Delta^-\cap\Delta^+)\leq 2$. Actually,  if either $\theta\in \Pi\cap \overline \Delta_0$, or $\theta\in  \Pi\cap \overline\Delta_1$, then $\hat{r}_{\theta}\Delta^-\cap\Delta^+=\{\theta\}$. If $\theta\in\Delta_1\backslash\overline\Delta_1$, then $\hat{r}_{\theta}\Delta^-\cap\Delta^+=\{\theta,2\theta\}$.

Thus we have a further observation which is a strengthened version of \cite[Propositoin 1.32]{CW}).

\begin{lemma}\label{basiclem1}
Let  $\Pi_{(i)}, i=1,2$ are two different simple root systems of $\mathfrak{g}$.  There exist a series of real and odd reflections $\hat{r}_1,\hat{r}_2,\cdots,\hat{r}_n$ such that $\hat w:=\hat{r}_n{\cdots}\hat{r}_2\hat{r}_1$ with $\hat w(\Pi_{(1)})=\Pi_{(2)}$. Moreover, for any $i\in\{1,2,\cdots,n\}$ there are at most two different roots between  $\hat{r}_{i+1}{\cdots}\hat{r}_2\hat{r}_1(\Delta^+_{(1)})$ and $\hat{r}_{i}{\cdots}\hat{r}_2\hat{r}_1(\Delta^+_{(1)})$.
\end{lemma}

\begin{proof} Set $\Delta^\pm_{(i)}$ to be the positive ({\textit{resp.}} negative) root system corresponding to $\Pi_{(i)}$ for $i=1,2$.  Due to  the difference between $\Pi_{(1)}$ and $\Pi_{(2)}$ we have  $\Delta_{(2)}^-\cap\Pi_{(1)}\neq\emptyset$. We verify the statement  by induction on $\#(\Delta_{(2)}^-\cap\Pi_{(1)})$. Say $\Delta_{(2)}^-\cap\Pi_{(1)}=\{\theta_1,\cdots,\theta_t\}$, and $\Pi_{(1)}=\{\theta_1,\cdots,\theta_t,\theta_{t+1},\cdots,\theta_{t+s}\}$.  
Consider $\hat{r}_1(\Pi_{(1)})$ for $\hat{r}_1:=\hat{r}_{\theta_1}$. Set $\Pi_{(21)}$ to be the new simple root system $\hat{r}_1(\Pi_{(1)})$. By the definition of real and odd reflections and the choice of $\theta_1$,  we can verify in a moment
\begin{align} \label{Basic Relation pm} \#(\Delta_{(2)}^-\cap\Pi_{(1)})-\#(\Delta_{(2)}^-\cap\Pi_{(21)})=1.
\end{align}
Now we begin to  verify the above formula, taking different cases of $\hat r_1$ into account. Note that $\{\theta_1,\cdots,\theta_t\}\subset \Delta_{(2)}^-$ while  $\{\theta_{t+1},\cdots,\theta_{t+s}\}\subset \Delta_{(2)}^+$. When $\hat r_1$ is a real reflection,  $\Pi_{(21)}=\{\hat r_1(\theta_1),\cdots,\hat r_1(\theta_t),\hat r_1(\theta_{t+1}),\cdots,\hat r_1(\theta_{t+s})\}$. Note that the row associated with $\theta_1$ of the Cartan matrix corresponding to $\Pi_1$ are non-positive integers except  the diagonal entry which is $2$ (\textit{cf}. \cite[Page 604-605]{K2}).  So the property that $\Pi_{(1)}$ is a fundamental root system implies that $\{\hat r_1(\theta_2),\cdots,\hat r_1(\theta_t)\}\subset \Delta_{(2)}^-$ while  $\{\hat r_1(\theta_{t+1}),\cdots,\hat r_1(\theta_{t+s})\}\subset \Delta_{(2)}^+$. In this case, the formula (\ref{Basic Relation pm}) is proved. When $\hat r_1$ is an odd reflection, or to say, $\theta_1$ is an isotropic simple odd root in $\Pi_{(1)}$,
$$\hat r_1(\Pi_{(1)})=\{-\theta_1\}\bigcup_{i\ne 1}\{\theta_i\mid (\theta_i,\theta_1)=0\}\bigcup_{j\ne 1}\{\theta_j+\theta_1\mid (\theta_j,\theta_1)\ne 0\}.$$
The property that $\Pi_{(1)}$ is a fundamental root system implies that $\{\overline\theta_2,\cdots,\overline\theta_t\}\subset \Delta_{(2)}^-$ while  $\{\overline\theta_{t+1},\cdots,\overline\theta_{t+s}\}\subset \Delta_{(2)}^+$ where $\overline\theta_i=\theta_i$ if $(\theta_i,\theta_1)=0$ or $\theta_i+\theta_1$ if $(\theta_i,\theta_1)\ne 0$. So in this case, the formula (\ref{Basic Relation pm}) is proved. Thus, we complete the verification of (\ref{Basic Relation pm}).

By the inductive assumption,  we can get  a series of real and odd reflections $\hat{r}_1,\hat{r}_2,\cdots,\hat{r}_n$ such that $\hat{r}_n{\cdots}\hat{r}_2\hat{r}_1(\Pi_{(1)})=\Pi_{(2)}$. As to the second statement, it follows from our inductive construction  and the observation mentioned  before the lemma.
\end{proof}

\subsection{Super Weyl groups}\footnote{Associating to any generalized root systems in Serganova's sense,  Sergeev and Veselov introduced in \cite{SerV} a notion of super Weyl groupoid. Heckenberger and Yamane introduced in \cite{HecY} a groupoid related to basic classical Lie superalgebras motivated by Serganova¡¯s work and in \cite{HecY1} the notion of the Weyl groupoid for Nichols algebras. Our super Weyl group has no direct relation with their notions.}
Denote by $\cb$  the set of all Borel subalgebras of $\ggg$ containing $\hhh$. By the definition of real and odd reflections, all of them can be regarded as transforms of the set $\cb$.
Actually, for a given $\hat r_{\theta_0}$, if it  is a real reflection, i.e. $\theta_0\in \overline\Delta_0^+\cup(\Delta_1^+\backslash\overline\Delta_1^+)$,  then $\hat r_{\theta_0}(B)= \hhh+\sum_{\theta\in \Delta(B)^+}\ggg_{s_{\theta_0}(\theta)}$ for any given Borel subalgebra $B=\hhh+\sum_{\theta\in \Delta(B)^+}\ggg_\theta$. If $\hat r_{\theta_0}$ is an odd reflection, then either $\theta_0$ or $-\theta_0$ is a simple isotropic odd root of some Borel subalgebra $B$. Denote by $\cb_{\theta_0}$ the set of such Borel subalgebras. Then $\cb_{\theta_0}$ is a subset of $\cb$ containing even number Borel subalgebras, say, $\cb_{\theta_0}=\{B_1,B_1',\cdots,B_t,B_t'\}$ such that the operation of $\hat r_{\theta_0}$ makes exchanged between $B_i$  and $B_i'$, $i=1,\cdots,t$. For any $B\in \cb\backslash \cb_{\theta_0}$, we define  an assignment of $\hat r_{\theta_0}$ which sends $B$ to itself.
Thus an odd reflection  $\hat r_\theta$  really can be regarded an transformation of $\cb$. Notice that for any isotropic root $\theta\in \overline\Delta_1^+$, there exists $w\in W$ such that $w(\theta)\in \Pi$ (\textit{cf}. \cite[Lemma 1.29]{CW}). Thereby  any isotropic $\theta$ always gives rise to an odd reflection.

\begin{defn}\label{superWeyl}  A subgroup of the transform group of $\cb$ generated by all real and odd reflections $\hat r_\theta$ is called the super Weyl group of $\ggg$, which is denoted by $\hat W$. 
\end{defn}
Proposition 1.32 of \cite{CW} shows an important property of $\hat W$ that any two fundamental systems of $\ggg$ are conjugate under $\hat W$.

It is readily seen that $\hat W$ contains  the usual Weyl group $W$ of $\ggg_0$, as a subgroup. Actually, for any $w\in W$, we regard $w$ as an element in $\hat W$, denoted by $\hat w$. Then the map $\varphi: w\mapsto \hat w$ defines  an group injective homomorphism from $W$ to $\hat W$. We only need to check that $\varphi$ is injective. This can be known that if $\hat w$ is an identity transform of $\cb$, then $\sigma(\Pi)=\Pi$. Owing to \cite[Theorem 10.3(e)]{Hum}, $w=1$. So $\hat W$ is generated by $W$ and  $\hat r_\theta$ for all $\theta\in \overline\Delta_1^+$, or to say, generated by  $\hat r_\theta$ for all $\theta\in  \Delta^+$.

We introduce a so-called extended fundamental root system $\tilde\Pi$ for $\ggg$ except type $A(m|n)$, $B(0|n)$, and $C(n)$. For those exceptional cases $A(m|n)$, $B(0|n)$ and $C(n)$, we set $\tilde\Pi=\Pi$. Recall the standard fundamental root system $\Pi$ for a basic classical Lie algebras is as follows (\textit{cf}. \cite[\S2.5]{K} or \cite[\S1.3]{CW}):
\begin{align}
&\mathfrak{spo}(2n|2m+1), m>0:
  & \Pi=\{\delta_1-\delta_2,\cdots,\delta_n-\varepsilon_1,\varepsilon_1-\varepsilon_2,\cdots,\varepsilon_{m-1}-\varepsilon_m,\varepsilon_m\};\cr
&\mathfrak{spo}(2n|2m), m\geq 2:
  &\Pi=\{\delta_1-\delta_2,\cdots,\delta_n-\varepsilon_1,\varepsilon_1-\varepsilon_2,\cdots,\varepsilon_{m-1}-\varepsilon_m,\varepsilon_{m-1}+\varepsilon_m\};\cr
&F(4):& \Pi=\{ {1\over 2}(\varepsilon_1+ \varepsilon_2+ \varepsilon_3+\delta), -\varepsilon_1, \varepsilon_1-\varepsilon_2, \varepsilon_2-\varepsilon_3\};\cr
&G(3):& \Pi=\{\delta+\varepsilon_1,   \varepsilon_2,  \varepsilon_3-  \varepsilon_2\};    \cr
&D(2,1,\alpha):& \Pi=\{ \varepsilon_1+ \varepsilon_2+\varepsilon_3, -2\varepsilon_2, -2\varepsilon_3\}.
\end{align}
We define $\tilde \Pi$ for $\ggg$ listed in  the above table:
\begin{align}
&\mathfrak{spo}(2n|2m+1), m>0:& \tilde\Pi= \{\alpha_0:=-2\delta_1\}\cup \Pi ;\cr
&\mathfrak{spo}(2n|2m), m\geq 2: & \tilde\Pi=\{\alpha:=-2\delta_1\}\cup \Pi ;\cr
&F(4):& \tilde\Pi=\{\alpha_0:=-\delta\}\cup\Pi;\cr
&G(3):& \tilde\Pi=\{\alpha_0:=-2\delta\}\cup\Pi;    \cr
&D(2,1,\alpha):& \tilde\Pi=\{\alpha_0:=-2\varepsilon_1\}\cup \Pi.
\end{align}
Then for the above non-exceptional cases, one has an extended standard Dynkin diagram corresponding to $\tilde\Pi$ which is such a diagram extending the standard Dynkin diagram associated with $\Pi$, by adding a vertex $\alpha_0$ connected to the first vertex of the standard Dynkin diagram.

Obviously, all real and odd reflections $\hat r_\theta$ satisfy $\hat r_\theta^2=\text{id}$. All of them are called {\sl super reflections}. A super reflection is called a simple one if it arises from a simple root of $\tilde\Pi$.
 We have further that  $\hat W$ is generated by all simple super reflections, in the same way as in the classical Lie algebras  case.

\begin{lemma} \label{simple ref gen}  The super Weyl group $\hat W$ is generated by the simple super reflections, this to say  $\hat W=\langle \hat r_\theta\mid \theta\in \tilde\Pi\rangle$.
\end{lemma}

\begin{proof} We have known that $\hat W$ is generated by $W$ and  $\hat r_\theta$ for  $\theta\in  \overline \Delta_1^+$.
 Note that $W$ is generated by the simple refections corresponding to the fundamental system of $\ggg_\bz$ which is contained in the extended standard simple root system $\tilde\Pi$. Hence
$W\subset  \langle \hat r_\theta\mid \theta\in \tilde\Pi\rangle$. On the other hand, checking  the standard Dynkin diagrams listed in \cite[\S1.3]{CW} (\textit{cf}. \cite[Proposition 1.5]{K2}), we know that there is only one isotropic simple root in the standard simple root system $\Pi$. So the remaining thing is to check that any two isotropic positive  roots $\gamma_i$, $i=1,2$ are conjugate by some $w\in W$.   By \cite[Lemma 1.29]{CW}, there exist $w_i\in W$ such that $w_i(\gamma_i)\in \Pi$. So $w_1(\gamma_1)=w_2(\gamma_2)$. Hence $\gamma_i$ are conjugate under $W$. The proof is completed.
\end{proof}

Thus $\hat W$ is actually a quotient of some Coxeter group associated with the generator system associated with $\tilde\Pi$. So one can study the super Weyl group by the Coxeter group theory. One can study the presentation of $\hat W$ (\textit{cf}. \cite{ChSh}), the reduced expressions and lengths of any elements, and the relation between the corresponding Hecke algebras, and representations of $\ggg$, which will be done  somewhere else.

\begin{example} In the case $\ggg=\gl(1|2)$, partially using the notations from \cite{CW} we have $\Delta_0=\{\pm(\epsilon_1-\epsilon_2)\}$, and $\Delta_1=\{\pm(\delta_1-\epsilon_1),\pm(\delta_1-\epsilon_2)\}$.  There are 6 Borel subalgebras corresponding to 6 systems of simple roots $\{\Pi_i\mid i=1,\cdots,6\}$:
\begin{align*}
&\Pi_1=\{\delta_1-\epsilon_1, \epsilon_1-\epsilon_2\},\;\; \Pi_2=\{\epsilon_1-\delta_1, \delta_1-\epsilon_2\},\;\;\Pi_3=\{\epsilon_1-\epsilon_2, \epsilon_2-\delta_1\}\cr
&\Pi_4=\{\delta_1-\epsilon_2, \epsilon_2-\epsilon_1\},\;\;\Pi_5=\{\epsilon_2-\delta_1, \delta_1-\epsilon_1\},\;\;\Pi_6=\{\epsilon_2-\epsilon_1, \epsilon_1-\delta_1\}
 \end{align*}
 With the numbers indexing those fundamental root systems, the super Weyl group can be regarded as a subgroup of $\mathfrak{S}_6$ presented precisely as follows:
$$\hat W=\langle \hat r_{\delta_1-\epsilon_1}=(12)(56),\; \hat r_{\epsilon_1-\epsilon_2}=((14)(25)(36)\rangle$$
which is the Coxeter group of type $G_2$.

\end{example}

\subsubsection{Standard reduced expressions and the standard lengths of elements in the super Weyl groups}
For a given element $\hat w$ of $\hat W$, set $\Pi'=\hat w(\Pi)$. Note that even if $\hat w\ne \id$, it could happen that $\Pi'=\Pi$. And
$\Pi'\ne \Pi$ if only if $\Phi'^{-}\cap \Pi\ne\emptyset$, where $\Phi'^{-}$ means the negative root system corresponding to $\Pi'$. Two elements $\hat w_i$, $i=1,2$ are called isogenic if $\hat w_1(\Pi)=\hat w_2(\Pi)$. Denote by {\sl the isogeny class} $\hat W(\Pi,\Pi')$ of $\hat w$ which consists of all $\hat\sigma$ sending $\Pi$ into $\Pi'=\hat w(\Pi)$. In $\hat W(\Pi,\Pi')$,  by Lemma \ref{simple ref gen} we can take a representative $\sigma$ such that
 \begin{itemize}
 \item[(1)] $\hat\sigma=\hat r_n\cdots\hat r_{1}$ for $\hat r_i=\hat r_{\theta_i}$.
 \item[(2)]  $\theta_{i}\in \Pi_{i-1}$ where $\Pi_{i-1}=\hat r_{i-1}\cdots \hat r_1(\Pi)$ for $i=1,\cdots,n$, and $\Pi_n=\Pi'$.
 \end{itemize}
Denote by $\ell_\Pi(\hat w)$ the smallest one among all possible $n$ as above, called {\sl the standard length of $\hat w$}.  The expression of $\hat \sigma=\hat r_n\cdots\hat r_{1}$  is called {\sl standard reduced} if it matches  the above (1) (2) with $n=\ell_\Pi(\hat w)$. In this case, we call $\sigma$ is a {\sl standard representative element of $\hat W(\Pi,\Pi')$}. Generally, $\hat W(\Pi,\Pi')$ is not necessarily to be a subgroup. But $\hat W(\Pi,\Pi)$ is a subgroup of $\hat W$.

For any given simple root system $\Pi'$ different from $\Pi$, we will introduce the distance between $\Pi'$ and $\Pi$, denoted by $\mathsf{d}(\Pi,\Pi')$.
\begin{defn}  Define $\mathsf{d}(\Pi,\Pi')=\#(\Delta'^-\cap\Delta^+) -\#(((\Delta'^-\cap\Delta^+)\cap \Delta_1^+ )\backslash ((\Delta'^-\cap\Delta^+)\cap \overline\Delta_1^+))$.
\end{defn}

Obviously,  $\mathsf{d}(\Pi,\Pi')=\mathsf{d}(\Pi',\Pi)$.
Furthermore, $\mathsf{d}(\Pi,\Pi')=0$ if only if $\Delta'^-\cap\Delta^+=\emptyset$, i.e. $\Delta'^+=\Delta^+$, or to say $\Pi=\Pi'$. And $\mathsf{d}(\Pi,\Pi')=1$ if and only if there is one simple super reflection $\hat r_\theta$ with $\theta\in \Delta'^-\cap \Pi$ such that $\Pi'=\hat r_\theta(\Pi)$ for $\theta\in \tilde\Pi$. If $\Pi'\ne \Pi$,  then $\mathsf{d}(\hat r_\theta(\Pi),\Pi')=\mathsf{d}(\Pi,\Pi')-1$ if only if  and  $\theta\in \Delta'^-\cap \Pi$.

\begin{prop} \label{length prop} For any given $\hat w\in \hat W$ and $\Pi'=\hat w(\Pi)$,  $\ell_\Pi(\hat w)= \mathsf{d}(\Pi,\Pi')$. Hence the expression of $\hat w$ in Lemma \ref{basiclem1} is standard reduced, while $\hat w$ is a standard representative element.
    \end{prop}

  \begin{proof} By Lemma \ref{basiclem1}, $\ell_\Pi(\hat w)\leq \mathsf{d}(\Pi,\Pi')$. So we only need to prove $n:=
  \ell_\Pi(\hat w)\geq d:=\mathsf{d}(\Pi,\Pi')$.
  Take a  standard representative element $\hat \sigma\in \hat W(\Pi,\Pi')$, and its reduced expression  $\hat\sigma=\hat r_n\cdots\hat r_{1}$ for $\hat r_i=\hat r_{\theta_i}$ and $n=\ell_\Pi(\hat w)$.
  We now prove the statement by induction on $n$.
 If $n=0$ then $\hat\sigma=\id$, and $\Pi=\Pi'$. We don't need do anything. Assume that  $n>0$, and that the statement holds for the case
 when less then $n$.  Notice that $\Pi'=\hat \sigma(\Pi)=\hat r_n\cdots\hat r_{1}(\Pi)=\hat r_n(\Pi_{n-1})$.
  By the inductive hypothesis, $n-1\geq \mathsf{d}(\Pi,\Pi_{n-1})$.
And
$\mathsf{d}(\Pi,\Pi')= \mathsf{d}(\Pi,\hat r_{\theta_n}(\Pi_{n-1}))= \mathsf{d}(\Pi,\Pi_{n-1})+1$ if and only if $\theta_n\in \Delta^+\cap (-\Pi')$.
In the other case, the standard reduced property of $\hat \sigma=\hat r_n\cdots\hat r_{1}$  implies that $\mathsf{d}(\Pi,\Pi')= \mathsf{d}(\Pi,\hat r_{\theta_n}(\Pi_{n-1}))= \mathsf{d}(\Pi,\Pi_{n-1})-1$. So in any case, $n-1\geq \mathsf{d}(\Pi,\Pi_{n-1})=d-1$, or $d+1$, thereby $n\geq d$. Combining the above arguments, we complete the proof for the case equal to $n$, thereby  the proof of the relation $\ell_\Pi(\hat w)\geq \mathsf{d}(\Pi,\Pi')$.
\end{proof}

\begin{lemma}\label{different pi}
  Maintain the notations as in Lemma \ref{basiclem1}.  Then, for $0<j<i\leq n$, $\hat{r}_i$ is different from $\hat r_j$ where  $n=\mathsf{d}(\Pi,\Pi')$.
     \end{lemma}
\begin{proof}
   For convenience we denote $\Pi=\Pi_0,~\Pi_{i-1}=\hat{r}_{i-1}{\cdots}\hat{r}_2\hat{r}_1(\Pi)$, one can choose $\theta_1\in(\Delta')^-\cap{\Pi_0}, \theta_i\in(\Delta')^-\cap{\Pi_{i-1}}$. Because for any integral number $0<i\leq{n}$ $\theta_{i-1}.\cdots,\theta_1\in\Delta_{i-1}^-$ but $\theta_i\in\Pi_{i-1}\subset\Delta_{i-1}^+$ so we can have that for any integral number $0<i\leq{n}$ $\hat{r}_i$ is different with $\hat{r}_{j}$ for any integral number $0<j<i$.
\end{proof}

\subsection{Distinguished elements}

\begin{defn} An element $\hat\sigma$ of $\hat W$ is called a distinguished element if $\hat\sigma$ is a standard representative element of $\hat W(\Pi, -\Pi)$.
\end{defn}

\begin{theorem}\label{longest element thm}
 Keep the notations as above. Then there exists a distinguished element $\hat w_0=\hat r_N \cdots \hat r_1\in \hat W$ with $\hat r_i=\hat r_{\theta_i}$ for a series of positive roots $\theta_i, i=1,\cdots,N$,  satisfying the following axioms:
\begin{itemize}
\item[(1)] $\theta_i$ is in the simple root system of the Borel subalgebra corresponding to  $\hat \sigma_{i-1}(\Delta^+)$.
Here $\hat\sigma_0:={\id}$, $\hat\sigma_{i-1}=\hat r_{i-1}\cdots \hat r_1$ $(i>0)$. And  $\hat\sigma_{i-1}(\Delta^+)$ stands for the positive root system after a series of odd and real reflections  $\hat r_1,\cdots,\hat r_{i-1}$.

\item[(2)]  $\hat w_0(\Pi)=-\Pi$ and $\hat w_0(\Delta^+)=-\Delta^+$.
\item[(3)] For $\hat \sigma_i=\hat r_i\cdots\hat r_1$, the following holds
\begin{align*}
\hat\sigma_i(\Delta^+)=\begin{cases}
(\hat\sigma_{i-1}(\Delta^+)\backslash\{\theta_i\})\cup\{-\theta_i\}, &\mbox{ if }\theta_i\in \hat\sigma_{i-1}(\Delta^+)\cap\overline\Delta_0;\cr
(\hat\sigma_{i-1}(\Delta^+)\backslash\{\theta_i\})\cup\{-\theta_i\}, &\mbox{ if }\theta_i\in \hat\sigma_{i-1}(\Delta^+)\cap\overline\Delta_1;\cr
(\hat\sigma_{i-1}(\Delta^+)\backslash\{2\theta_i,\theta_i\})\cup\{-2\theta_i,-\theta_i\}, &\mbox{ if }\theta_i\in \hat\sigma_{i-1}(\Delta^+)\cap
(\Delta_1\backslash \overline\Delta_1).
\end{cases}
\end{align*}
\item[(4)] $N=\#\Delta^+-\#(\Delta_1^+\backslash\overline\Delta^+_1)$.
\item[(5)] $\Delta^+=\{ \theta_i\mid i=1,2\cdots,N\}\cup \{2\theta_i\mid \theta_i \in (\Delta_1^+\slash\overline\Delta_1^+)\cap\sigma_{i-1}(\Delta^+), i=1,\cdots,N\}$.
\end{itemize}
\end{theorem}

\begin{proof} We apply Lemma \ref{basiclem1} to the case when $\Pi'=-\Pi$. By Lemma \ref{different pi},
we know that $\hat r_{\theta_i}\ne \hat r_{\theta_j}$ whenever $i\ne j$. Furthermore,  both $\theta_i$ and $2\theta_i$ never occur simultaneously in $\{\theta_1,\theta_2,\cdots\}$ if  $\theta_i\in \Delta_1^+\backslash \overline\Delta^+_1$. So,  the choice of those $\theta_i$ from the arguments in the proof of Lemma \ref{basiclem1}  implies that all positive roots have to be involved when we finish the operation of  changing $\Pi$ into $-\Pi$ (note that we identify the reflections arising from  $\overline\Delta^+_0$ with  the ones from $\Delta_1^+ \backslash \overline\Delta^+_1$). From the above arguments, the statements (1)-(4) follow. Furthermore, combining with  Proposition \ref{length prop} we know $\hat w_0$ is a distinguished element. The proof is completed.
\end{proof}

\begin{remark}
 By Proposition \ref{length prop} and Theorem \ref{longest element thm},
 we know that the longest standard length of $\hat W$ is $N$, and  $\hat w_0$ is of the longest standard length $N$.  This element plays a role as important as in the longest element $\omega_0$ in the usual Weyl group $W$ of a classical Lie algebra. Be careful that the longest standard length may not coincide with the longest length in the corresponding Coxeter system.
\end{remark}

\section{Graded restricted module categories}
Keep the notations and assumptions as in the previous two sections. Set $\bbz\Pi$ to be the free abelian group spanned by $\Pi$. Then $\ggg=\hhh+\sum_{\alpha\in \Delta}\ggg_\alpha$ is naturally a $\bbz\Pi$-graded Lis superalgebra. We can consider a so-called $\bbz\Pi$-graded $\ggg$-module category which consists of the objects satisfying the $\bbz\Pi$-graded structure compatible with $\bbz\Pi$-graded structure of $\ggg$, i.e.
for $M=\sum_{\lambda\in \bbz\Pi}M_\lambda$, $g_\alpha\cdot M_\lambda\subset M_{\alpha+\lambda}$.

Set  $$Y:=X(\frakt).$$
By Lemma \ref{character gp}, we can extend a $\bbz\Pi$-graded module category of $\ggg$ to some $Y$-graded module category of $\ggg$.
\subsection{$(U_0(\ggg),\frakt)$-module categories}
  By definition, a rational $\frakt$-module $V$  means a $Y$-graded vector space admissible with rational $\frakt$-action, i.e.
$V=\sum_{\mu\in Y}V_\mu$, where
$V_\mu=\{v\in V\mid t\cdot v=\mu(t) v,\; \forall t\in\frakt \}$.
Recall that the automorphism group of $\g$ contains a closed subgroup $G_\ev$.
For $\mu\in Y$,  its
differential $\textsf{d}\mu: \h\rightarrow \bbk$, is a homomorphism
of restricted Lie algebras and satisfies
${\textsf{d}}\mu(H^{[p]})=({\textsf{d}}\mu(H))^{[p]}$.   This means that
${\textsf{d}}\mu\in \Lambda:=\{\lambda\in \hhh^*\mid \lambda(H)^p=\lambda(H^{[p]}), \forall H\in \hhh \}$. The map $\varphi: \mu\mapsto
{\textsf{d}}\mu$ has kernel $pY$. And this induces a
bijection $Y/pY\cong \Lambda$.  We may identify $Y\slash pY$ with
$\Lambda$. Sending $\mu\in Y$ to $\overline\mu\in
\Lambda=Y\slash pY$, we write ${\textsf{d}}\mu(H)$
directly as $\overline\mu(H)$ without any confusion, and call them
\textit{restricted weights}. (Sometimes, ${\textsf d}\mu$ and $\mu$ are not discriminated in use if no confusion happens in context.)

 Naturally, $U_0(\g)$
 and its canonical subalgebras which will be used later
become rational $\frakt$-modules with the action denoted by
$\Ad(T)a$ for $T\in \frakt$ and $a\in U_0(\g)$.

Let us introduce the full subcategory $(U_0(\g),\frakt)\hmod$ of the
$U_0(\g)$-module category $U_0(\g)\hmod$:

\begin{definition}\label{UGTMOD}\footnote{
With the same spirit, there are other formulations different from the category
$(U_0(\g),\frakt)\hmod$ used here, such as the $G_1\frakt$-module category for the first Frobenusou kernel (\textit{cf}. \cite{J01} and
\cite{J3}), the $\bbz$-graded $U_0(\g)$-module category (\textit{cf}.
\cite[\S11]{J1}), and the $u(\g)\#\Dist(\frakt)$-module category
(\textit{cf}. \cite{ln}, see Remark \ref{FOOT}(2)). Here we follow the formulation used in \cite{SZha}.} The category
$(U_0(\g),\frakt)\hmod$ is defined as such a category whose objects
are finite-dimensional $\bbk$-superspaces endowed with both
$U_0(\g)$-module and rational $\frakt$-module structure satisfying the
following compatibility conditions for $V\in (U_0(\g),\frakt)\hmod$:

(i) The action of $U_0(\h)$ coincides with the action of
$\Lie(\frakt)$ induced from $\frakt$.

(ii) For $a\in U_0(\g)$, $T\in \frakt$, and $v\in V$: $T(av)=(\Ad
T(a))Tv$.

The morphisms of $(U_0(\g),\frakt)\hmod$ are defined to be linear maps
of $\bbk$-superspaces acting as both $U_0(\g)$-module homomorphisms, and
rational $\frakt$-module homomorphisms.
\end{definition}

\begin{remark} \label{FOOT} (1) Let $u$ be a Hopf subalgebra of $U_0(\g)$ (or $U(\ggg)$) with adjoint $\hhh$-module structure.
We can define a category of $(u,\frakt)\hmod$ in the same way as
above, of which each object is simply  called a  $\hat u$-module.

(2)  The $\hat u$-module category can be realized a module category
of the precise Hopf algebra $\hat u=u\#\Dist(\frakt)$, where $\Dist(\frakt)$ denote the distribution algebra of $\frakt$ as defined
in \cite{ln}.
\end{remark}

\subsubsection{Baby Verma modules} \label{baby Verma mod}

  Naturally, $U_0(\g)$, $U_0(\bbb)$ become $\frakt$-modules, compatible with the restricted $\hhh$-module structure. For stressing such structure, we denote them by $\hat U_0(\ggg)$ and $\hat U_0(\bbb)$ respectively.
  Hence for any $\mu\in Y$, one naturally has a one-dimensional $\hat U_0(\bbb)$-module $\bbk_\mu$ with trivial $\nnn^+$-action, $\hhh$-action of multiplication  via $\mathsf{d}(\mu)$, and $\frakt$-action of multiplication via  $\mu$. Therefore, we can endow with $\hat U_0(\g)$-module structure on
     $$\hat Z(\mu)=\hat U_0(\g)\otimes_{\hat U_0(\bbb)}\bbk_\mu.$$
     Once the parity of the one-dimensional space $\bbk_\mu$ is given{\footnote{For the case of restricted representation categories or $(u,\frakt)$-module categories, one can give parities for  weight spaces similar to \cite[\S6]{CLW}}}, say $\epsilon\in \bbz_2=\{\bz,\bo\}$, then the super-structure of $\hat Z(\mu)$ is determined by the super structure of $\hat U(\nnn^-)=\hat U(\nnn^-)_\bz \oplus \hat U(\nnn^-)_\bo$ together with  $\epsilon$  as follows
     $$\hat Z(\mu)=\hat Z(\mu)_\bz\oplus \hat Z(\mu)_\bo, \mbox{ where } \hat Z(\mu)_{\delta+\epsilon}=\hat U_0(n^-)_\delta\otimes\bbk_\mu \mbox{ for }\delta\in\bbz_2.$$
     Note that 
 $Y$ is a poset with partial order "$\leq$": $\lambda\leq \mu$ if and only if $\mu-\lambda \in \bbz_{\geq 0}\Pi$. By the definition, $\hat Z(\mu)$ is a highest weight module with highest weight $\mu$.
  As $U_0(\ggg)$-modules, $\hat Z(\mu)$ is isomorphic to $Z(\bar\mu):=U_0(\ggg)\otimes_{U_0(\bbb)}\bbk_{\bar\mu}$. Both of them share the same super space structure.

     Generally, we can define for any $\lambda\in \Lambda$,
     $$Z(\lambda)=U_0(\ggg)\otimes_{U_0(\bbb)}\bbk_{\lambda},$$
      called a baby Verma module. So there is a natural forgetful functor $\mathcal{F}$ from $(U_0(\ggg),\frakt)\hmod$ to $U_0(\ggg)\hmod$ which  sends $\hat Z(\mu)$ to $Z(\textsf{d}\mu)$ for any $\mu\in Y$.

 Following the same  arguments as in \cite[Proposition 10.2]{J1} or  \cite[Lemma 2.1]{SZha}, we can easily obtain
 \begin{lemma} In the categories $U_0(\ggg)\hmod$ and $(U_0(\ggg),\frakt)\hmod$, the following statements hold.
 \begin{itemize}
 \item[(1)]  Both of $Z(\lambda)$ and $\hat Z(\mu)$ for any $\lambda\in \Lambda$ and $\mu\in Y$ admit unique maximal submodules respectively. The unique irreducible quotients are denoted by $L(\lambda)$ and $\hat L(\mu)$ respectively.
 \item[(2)] The above $\hat L(\mu)$ belongs to $(U_0(\ggg),\frakt)\hmod$.
 \end{itemize}
  \end{lemma}

\subsubsection{Isomorphism classes of irreducible modules in $U_0(\ggg)\hmod$ and $(U_0(\ggg),\frakt)\hmod$}

\begin{proposition}
\begin{itemize}
\item[(1)]  The iso-classes of irreducible modules in
$(U_0(\g),\frakt)\hmod$ are in one-to-one correspondence with
$Y$. Precisely, each simple objects in $(U_0(\g),\frakt)\hmod$
is isomorphic to $\hat L(\mu)$ for $\mu\in Y$.

\item[(2)] $\hat Z(\mu)|_{U_0(\g)}\cong Z(\overline \mu)$, and $\hat
L(\mu)|_{U_0(\g)}\cong L(\overline \mu)$. Furthermore, sending
$\mu\in Y$ to $\overline\mu\in \Lambda$
 gives rise to the surjective map  $\hat
L(\mu)\mapsto L(\overline \mu)$ from the set of iso-classes of
simple objects of $(U_0(\g),\frakt)\hmod$ and to those of $U_0(\g)\hmod$.
\end{itemize}
\end{proposition}



By the above proposition, good understanding of
$(U_0(\g),\frakt)\hmod$ can provide us  sufficient  information on
restricted simple modules of $\g$.

\vskip5pt
In this paper, we will {\sl focus on  $(U_0(\g),\frakt)\hmod$. For the simplicity of notations, we simply write $\mathcal{C}$ } for this category in the most time of the rest of the paper.

\subsubsection{Base change of $\mathcal{C}$}
   Keep the above notations. In particular, let $\mathfrak{g}$ be a given basic classical Lie superalgebra over $\bbk$ with $\Lie(G)=\ggg$ for an algebraic supergroup $G$ of Chevalley type, $Y=X(\mathfrak{T})$.   Let $A$ be a commutative $\bbk$-algebra with unity element.  We will extend the category $(U_0(\ggg),\mathfrak{T})\hmod$ over $\bbk$ to the category $(U_0(\ggg)\otimes A,\mathfrak{T})\hmod$ in the same way as in \cite{J2}, simply denoted by $\mathcal{C}_A$ over $A$,  which satisfies the same axioms as in Definition \ref{UGTMOD}(i)(ii) together with a given $\bbk$-algebra homomorphism $U(\hhh)\rightarrow A$.
      For this we  first introduce
  $U$ which denotes  the quotient of $U(\ggg)$ by the ideal generated by  $x^p_\alpha-x_\alpha^{[p]}$ for $x_\alpha\in \ggg_\alpha$,  $\forall \alpha\in \Delta_0$.
  Then we have  an isomorphism of vector spaces
  $$
U_0(\nnn^-)\otimes U(\hhh)\otimes U_0(\nnn^+)\rightarrow U.$$
We shall write $U^0$ for the image of $U(\hhh)$ in $U$. This image is isomorphic to $U(\hhh)$.

    \begin{defn}\label{Ext A-Mod}
    For a given $\bbk$-algebra homomorphism $\pi: U^0\rightarrow A$, we define the extended category $\mathcal{C}_A$ of  $\mathcal{C}$ which consists of  the objects satisfy the condition (1) (2) and (3) as below: for  $M=\sum_{\mu\in Y}M_\lambda \in{\text{obj}(\mathcal{C}_A)}$
\begin{itemize}
\item[(1)] $M$ is a $\hat U_0(\ggg)\otimes{A}$ module with  $Y$-graded structure admissible with $\mathfrak{T}$-action, as shown as formulated in  Definition \ref{UGTMOD}(i)(ii).
\item[(2)] $A\cdot{M_\mu}\subset M_\mu$.
\item[(3)] The action of $\mathfrak{h}$ on each graded component $M_{\mu}$ is diagonalisable and compatible with $Y$-graded structure and $\pi$, this is to say for $Z:=\Hom_\bbk(\hhh,A)$,
   $$M_{\mu}=\bigoplus_{\nu\in Z}M_\mu^{(\nu)}$$
   satisfying that   $M_{\mu}^{(\nu)}\neq{0}$ implies  ${\nu=\pi+\mathsf{d}(\mu)}$.
\end{itemize}
\end{defn}

Given a commutative $\bbk$-algebra $A$,  a $\bbk$-algebra homomorphism $\pi:U^0\rightarrow A$, and $\lambda\in Y$,  by Remark \ref{FOOT} we can  define a baby Verma module  $\hat Z_A(\lambda)$ in $\mathcal{C}_A$
 $$\hat Z_A(\lambda)=\hat U\otimes_{\hat U^0\hat U_0(\mathfrak{\nnn}^+)}A_{\lambda}$$
 where $A_{\lambda}=A\otimes\bbk_{\lambda}$ is endowed with a $\hat U^0\hat U_0(\mathfrak{\nnn}^+)$-module structure with $H$-action for $H\in\hhh$ as multiplication  by $\pi(H)+\mathsf{d}(\lambda)(H)$ and trivall $\nnn^+$-action, and  $A_\lambda$ admits a trivial $\nnn^+$-module structure and shares the same parity structure with $\bbk_\lambda$.

\section{Baby Verma modules and their twists in $\mathcal{C}_A$}\label{baby Verma hom}

Keep the notations as in the previous sections. 
 For $\sigma\in \hat W$, we have a new Borel subalgebra $\sigma(\bbb)=\hhh\oplus \sigma(\nnn^+)$ where $\sigma(\nnn^+):=\sum_{\alpha\in \sigma(\Delta^+)}\ggg_\alpha$. Obviously, $\sigma(\nnn^+)$ is still a restricted Lie super subalgebra of $\ggg$.  Hence we have a twisted baby Verma module for $\lambda\in Y$
$$\hat Z_A^\sigma(\lambda):=\hat U\otimes_{\hat U^0\hat U_0(\sigma(\nnn^+))} A_\lambda.$$
In this section, we will make some preparation for constructing Jantzen filtration for baby Verma modules in the category  $\mathcal{C}$. More precisely speaking, we need to describe (twisting) homomorphisms in $\mathcal{C}_A$ between a given twisted Verma modules $\hat Z_A^\sigma(\lambda)$  and some of its twisted ones via super reflection $\hat Z_A^{\hat r\cdot \sigma}(\lambda^{\hat r})$ (see (\ref{lambda r}) for $\lambda^{\hat r}$). For this, we only need to understand such homomorphisms from a standard baby Verma modules $\hat Z(\lambda)$ to its super-simple-reflection twist $\hat Z^{\hat r}(\lambda^{\hat r})$. We will exploit the arguments of constructing homomorphisms between twisted baby Verma modules  in \cite{J2} to the super case, for which we need to deal with extra complicated situations arising from odd roots and odd reflections.

\subsection{Homomorphisms between baby Verma modules and their twists via super reflections}\label{varity of varphi}
Now we begin with $\hat Z_A(\lambda)$. For a given super simple reflection $\hat r:=\hat r_\alpha$ for $\alpha\in \Pi$, we will define  different twisting homomorphisms  $\varphi$ from $\hat Z_A(\lambda)$ to $\hat Z_A^{\hat r}(\lambda^{\hat r_\alpha})$, according to the different type for $\alpha$ (the definition of $\lambda^{\hat r_\alpha}$ is given in \ref{lambda r}). We construct those homomorphisms, according to three different cases of super simple reflections.

Recall the Chevalley bases introduced in \S\ref{Chevalley Basis}, and the Weyl vector  $\rho=\rho_0-\rho_1$ where $\rho_0={1\over 2}\sum_{\theta\in \Delta^+_0} \theta$ and  $\rho_1={1\over 2}\sum_{\theta\in \Delta^+_1} \theta$. We set $\hat r_\alpha\rho_0:={1\over 2}\sum_{\theta\in\hat r_\alpha(\Delta)^+_0}\theta$ and $\hat r_\alpha\rho_1:={1\over 2}\sum_{\theta\in\sigma_i(\Delta)^+_1}\theta$
for  $\hat r_\alpha(\Delta)^+=\hat r_\alpha(\Delta)_0^+\cup\hat r_\alpha(\Delta)_1^+$  the positive root set associated with the simple root system $\hat r_\theta(\Pi)$. Then for $\lambda\in Y$, we set
\begin{align}\label{lambda r}
\lambda^{\hat r_\alpha}:=
(\lambda-(p-1)(\rho_0-\hat r_\alpha\rho_0)-(\rho_1-\hat r_\alpha\rho_1)).
\end{align}

In this section, we will give and investigate a homomorphism
$$\varphi: \hat Z_A(\lambda)\rightarrow \hat Z_A^{\hat{r}_{\alpha}}(\lambda^{\hat r_\alpha})$$
whose exact definition will be presented sequently, according to the different types for  $\alpha$.
 \vskip5pt
 \subsubsection{(Case 1): $\alpha\in \Pi\cap \overline \Delta_0$  an even simple root, but half of which is not a root}

  In this case, $\lambda^{\hat r_\alpha}=\lambda-(p-1)\alpha$. We can construct a homomorphism $\varphi$ by the same way as in \cite[Lemma 3.4]{J2}:
$$\varphi: \hat Z_A(\lambda)\rightarrow \hat Z_A^{\hat{r}_{\alpha}}(\lambda-(p-1)\alpha))$$
given by $\varphi(1\otimes 1)=X_\alpha^{p-1}\otimes 1$,
and a converse homomorphism
$$\varphi':~Z_A^{\hat{r}_{\alpha}}(\lambda-(p-1)\alpha)
\rightarrow{Z_A(\lambda)}$$
given by  $\varphi'(1\otimes 1)=Y_{\alpha}^{p-1}\otimes 1$. Here and thereafter $1\otimes 1$ stands for the highest weight vector in its  baby Verma module in $\mathcal{C}_A$.

Then we have the following result by the same arguments as in modular representations of reductive Lie algebras,  which is due to Jantzen.
\begin{lemma} (\cite[Lemma 3.5]{J2}) \label{lemma case 1}
Assume $A=\bbk$ and $\pi(\mathfrak{h})=0$. Take $d\in \{0,1,\cdots, p-1\}$ such that $d=\langle\lambda,\chi_\alpha\rangle(\mbox{mod } p)$ where $\alpha$ is as in (Case 1). Then the following statements hold.
\begin{itemize}
\item[(1)] If $d=p-1$, then $\varphi$ and $\varphi'$ are isomorphisms.
\item[(2)] If $d<p-1$,   $\ker\varphi=\im\varphi'\cong{\coker\varphi}$ and $\ker\varphi'=\im\varphi\cong{\coker\varphi'}$.

Furthermore,
 we have an exact sequence
$${\cdots}\rightarrow{\hat Z(\lambda-(p+d+1)\alpha)}\rightarrow{\hat Z(\lambda-p\alpha)}\rightarrow{\hat Z(\lambda-(d+1)\alpha)}\rightarrow{\ker\varphi}\longrightarrow 0.$$
\end{itemize}
\end{lemma}

\vskip10pt
\subsubsection{(Case 2): $\alpha\in \Pi \cap \overline \Delta_1$  a simple isotropic odd root}\label{subsec case 2}

 In this case, $\lambda^{\hat r_\alpha}=\lambda-\alpha$. We can construct a homomorphism
 $$\varphi: \hat Z_A(\lambda)\rightarrow \hat Z^{\hat r_\alpha}(\lambda-\alpha)\mbox{ and } \;\;\;\varphi': Z_A^{\hat{r}_{\alpha}}(\lambda-\alpha)\rightarrow Z_A(\lambda)  $$
 given by $\varphi(1\otimes1)=X_\alpha\otimes 1$ and  $\varphi'(1\otimes 1)=Y_{\alpha}\otimes1 $ respectively.

Set $v_0=1\otimes 1, v_1=Y_{\alpha}\otimes 1 \in  \hat Z_A(\lambda)$ and $v'_0=1\otimes 1, v'_1=X_{\alpha}\otimes 1 \in \hat Z^{\hat r_\alpha}(\lambda-\alpha)$.
By a straightforward  calculation£¬
we have
\begin{align}\label{case 2 comp}
&\varphi(v)=\begin{cases}
v'_1,&{v=v_0},\cr
(\pi(H_{\alpha})+\lambda(H_{\alpha}))v_0',&{v=v_1};
\end{cases}\cr
&\varphi'(v')=\begin{cases}
v_1,&{v'=v'_0},\cr
(\pi(H_{\alpha})+\lambda(H_{\alpha})){v_0},&{v'=v'_1}.
\end{cases}
\end{align}

We have the following lemma.

\begin{lemma}\label{lemma case 2}
Assume that $\alpha$ is as in (Case 2), $A=\bbk$ and $\pi(\mathfrak{h})=0$. Still take $d\in \{0,1,\cdots, p-1\}$ such that $d=\langle\lambda,\chi_\alpha\rangle(\mbox{mod } p)$. The following statements hold.
 \begin{itemize}
 \item[(1)] If $d\ne 0$, i.e. $\langle\lambda, \chi_{\alpha}\rangle\not\equiv 0 (\mbox{mod }p)$, then $\varphi$ and $\varphi'$ are isomorphisms.
 \item[(2)] If $d=0$, i.e. $\langle\lambda, \chi_\alpha\rangle\equiv 0 (\mbox{mod }p)$, then
$$\ker\varphi=\im\varphi'\cong{\coker\varphi};\;\;\; \ker\varphi'=\im\varphi\cong{\coker\varphi'}.$$

Furthermore, we have an exact sequence$${\cdots}\rightarrow{\hat Z(\lambda-2\alpha)}\rightarrow{\hat Z(\lambda-\alpha)}\rightarrow{\ker\varphi}\rightarrow0$$ in $\mathcal{C}$.
    \end{itemize}
\end{lemma}

\begin{proof} (1)  Set $$\mathfrak{m}:=\bigoplus_{\beta\in \Delta^+,\beta\neq\alpha}{\mathfrak{g}_{-\beta}}.$$
Then $\mathfrak{m}$ is a restricted Lie super subalgebra of $\nnn^-$. Take a basis $B$ of $U_0(\mathfrak{m})$. Then $\{mv_i\mid i=0,1; m\in{B}\}$ constitute a basis of $\hat Z(\lambda)$. And $\{mv'_i\mid i=0,1; m\in{B}\}$ constitute a basis of $\hat Z^{\hat r_\alpha}(\lambda-\alpha)$.
Thus, by (\ref{case 2 comp} we know that $\varphi$ and $\varphi'$ are isomorphisms when $\langle\lambda, \chi_\alpha\rangle\not\equiv 0 (\mbox{mod }p)$. The statement (1) is proved.

(2) Now suppose $\langle\lambda, \chi_\alpha\rangle\equiv 0 (\mbox{mod }p)$. Then by (\ref{case 2 comp}) again, we have
\begin{align}\label{image varphi case 2}
&\ker\varphi=\bbk\mbox{-span}\{Bv_1\},\;\im\varphi=\bbk\mbox{-span}\{Bv'_1\},\cr
&\ker\varphi'=\bbk\mbox{-span}\{Bv'_1\},\;\im\varphi'=\bbk\mbox{-span}\{Bv_1\}.
\end{align}
 So we have
$$\coker\varphi=\frac{\hat Z^{\hat v_\alpha}(\lambda-\alpha)}{\im\varphi}=\frac{\hat Z{\hat v_\alpha}(\lambda-\alpha)}{\ker\varphi'}\cong{\im\varphi'}.$$
Similarly,  $\coker\varphi'\cong{\im\varphi}.$  The satement (2) is proved.

(3) We continue to consider  homomorphisms $\varphi_1:{\hat Z(\lambda-2\alpha)}\rightarrow{\hat Z(\lambda-\alpha)}$ via $\varphi_1(1\otimes 1)={Y_{\alpha}\otimes 1}$; and $\varphi_2:{\hat Z(\lambda-\alpha)}\rightarrow{\hat Z(\lambda)}$ via $\varphi_2(1\otimes 1)={Y_{\alpha}\otimes 1}$. Then we have an exact sequence
$${\cdots}\rightarrow{\hat Z(\lambda-2\alpha)}{\overset{\varphi_2}{\longrightarrow}}{\hat Z(\lambda-\alpha)}{\overset{\varphi_1}{\longrightarrow}}{\ker\varphi}\rightarrow 0.$$
We complete  the proof.
\end{proof}

\vskip10pt
\subsubsection{(Case 3): $\alpha\in \Pi\cap(\Delta_1\backslash\overline\Delta_1)$}\label{subsec case 3}

 In this case, $2\alpha\in \Delta_0$,  and $\hat r_{\alpha}=\hat r_{2\alpha}$. Then $\lambda^{\hat r_\alpha}=\lambda-(2p-1)\alpha$.
We can define  homomorphisms
$$\varphi:\hat Z_A(\lambda)\rightarrow{\hat Z_A^{\hat r_\alpha}(\lambda-(2p-1)\alpha)}$$
via $\varphi(1\otimes 1)={X_\alpha^{2p-1}\otimes 1}$ and
$$\varphi':\hat Z_A^{\hat r_{\alpha}}(\lambda-(2p-1)\alpha)\rightarrow{\hat Z_A(\lambda)}$$
via $\varphi'(1\otimes 1)={Y_\alpha^{2p-1}\otimes 1}$.

For $0\leq{i}\leq 2p-1$, set
$v_0:=1\otimes1_{\lambda}$, $v_i:=Y_\alpha^i\otimes 1\in \hat Z_A(\lambda)$  and  $v'_0:=1\otimes 1$, $v'_i=X_\alpha^i\otimes 1\in \hat Z_A^{r_\alpha}(\lambda-(2p-1)\alpha)$ for $i=1,\cdots, 2p-1$.
Simple calculation show  for $i>0$ that
\begin{align*} X_\alpha. Y_\alpha^i\otimes1=\begin{cases}
-iY_\alpha^{i-1}\otimes1,&\mbox{ when }i \mbox{~is~even}\cr
(\pi(H_\alpha)+\lambda(H_\alpha)-(i-1))Y_\alpha^{i-1}\otimes1,&\mbox{ when } i~\mbox{ is~odd}
\end{cases}
\end{align*}
in $\hat Z_A(\lambda)$
and
\begin{align*}
Y_\alpha.X_\alpha^i\otimes1=\begin{cases}
iX_\alpha^{i-1}\otimes1, &\mbox{ when }i \mbox{ is~even}\\
(\pi(H_\alpha)+\lambda(H_\alpha)+i+1) X_\alpha^{i-1}\otimes1, &\mbox{ when }i \mbox{~is~odd}
\end{cases}
\end{align*}
in $\hat Z^{\hat r_{\alpha}}_A(\lambda-(2p-1)\alpha)$.
Thus we can get the following by induction on $i$:
\begin{align}\label{case 3 comp}
&\varphi(v_i)=\begin{cases}
\prod_{j=1}^{{i\over 2}}(\pi(H_\alpha)+\lambda(H_\alpha)+2(p-j+1))\prod_{j=1}^{i\over 2}(2(p-j))v'_{2p-1-i},&\mbox{ when } i \mbox{~is~even}\\
\prod_{j=1}^{{i+1\over 2}}(\pi(H_\alpha)+\lambda(H_\alpha)+2(p-j+1))\prod_{j=1}^{i-1\over 2}(2(p-j))v'_{2p-1-i},&\mbox{ when }i \mbox{~is~odd}
\end{cases}\cr
&\varphi'(v'_i)=\begin{cases}
(-1)^{\frac{i}{2}}\prod_{j=1}^{\frac{i}{2}}(\pi(H_\alpha)+\lambda(H_\alpha)-2(p-j))\prod_{j=1}^{\frac{i}{2}}(2(p-j))v_{2p-1-i},&\mbox{ when } i~\mbox{ is~even}\\
(-1)^{\frac{i-1}{2}}\prod_{j=1}^{\frac{i+1}{2}}(\pi(H_\alpha)+\lambda(H_\alpha)-2(p-j))\prod_{j=1}^{\frac{i-1}{2}}(2(p-j))v_{2p-1-i},&\mbox{ when } i\mbox{~is~odd}.
\end{cases}
\end{align}

Then we have the following
\begin{lemma}\label{lemma case 3}
Assume that $A=\bbk$ and $\pi(\mathfrak{h})=0$. Take $d\in \{0,1,\cdots, p-1\}$ such that $\langle\lambda, \chi_\alpha\rangle\equiv d(\mbox{mod } p)$ where $\alpha$ is as in (Case 3).
\begin{itemize}
    \item[(1)]  $\ker\varphi=\im\varphi'\cong \coker\varphi$ and
$\ker\varphi'=\im\varphi\cong{\coker\varphi'}$.

\item[(2)]  Furthermore, we have an exact sequence
\begin{itemize}
\item[(i)] when $d$ is an odd number,
$${\cdots}\rightarrow{\hat Z(\lambda-(3p+d+1)\alpha)}\rightarrow{\hat Z(\lambda-2p\alpha)}
\rightarrow{\hat Z(\lambda-(p+d+1)\alpha)}\rightarrow{\ker\varphi}\rightarrow 0;$$
\item[(ii)] When $d$ is an  even number,
$${\cdots}\rightarrow{\hat Z(\lambda-(2p+d+1)\alpha)}\rightarrow{\hat Z(\lambda-2p\alpha)}
\rightarrow{\hat Z(\lambda-(d+1)\alpha)}\rightarrow{\ker\varphi}\rightarrow 0;$$
\end{itemize}
\end{itemize}
\end{lemma}

\begin{proof}  Still consider the restricted Lie super subalgebra of $\nnn^-$: $$\mathfrak{m}=\bigoplus_{\beta\in{\Delta},\beta\not\in\bbn\alpha}{\mathfrak{g}_{-\beta}}$$
 where $\bbn$ stands for the positive integer set.  Take a basis $B$ of $U_0(\mathfrak{m})$. Then $\{mv_i\mid i=0,1,\cdots,2p-1; m\in{B}\}$ constitute a basis of $\hat Z(\lambda)$. And $\{mv'_i\mid i=0,1,\cdots, 2p-1; m\in{B}\}$ constitute a basis of $\hat Z^{\hat r_\alpha}(\lambda-\alpha)$.

Under the assumption, the formulas in (\ref{case 3 comp}) become
\begin{align*}
\varphi(v_i)=\begin{cases}
\prod_{j=1}^{\frac{i}{2}}(d+2-2j))\prod_{j=1}^{\frac{i}{2}}(2(p-j))v'_{2p-1-i},&\mbox{ when }i \mbox{~is~even}\cr \prod_{j=1}^{\frac{i+1}{2}}(d+2-2j))\prod_{j=1}^{\frac{i-1}{2}}(2(p-j))v'_{2p-1-i}, &\mbox{ when } i \mbox{~is~odd}
\end{cases}
\end{align*}
and
\begin{align*}
\varphi'(v'_i)=\begin{cases}
(-1)^{\frac{i}{2}}\prod_{j=1}^{\frac{i}{2}}(d+2j)\prod_{j=1}^{\frac{i}{2}}(2(p-j))v_{2p-1-i},&\mbox{ when } i \mbox{~is~even}\cr
(-1)^{\frac{i-1}{2}}\prod_{j=1}^{\frac{i+1}{2}}(d+2j)\prod_{j=1}^{\frac{i-1}{2}}(2(p-j))v_{2p-1-i},& \mbox{ when }i \mbox{~is~odd}.
\end{cases}
\end{align*}
From the above formula, by a direct computation we have the following observations
 \begin{itemize}
\item [(i)] $\im\varphi'$ is spanned by $\{mv_i\mid{i} \geq p+d+1,m\in{B}\}$ if $d$ is odd, and spanned by $\{mv_i\mid i=1,\cdots,2p-1,m\in{B}\}$ if $d=0$; by  $\{mv_i\mid{i} \geq d+1,m\in{B}\}$ if $d$ is a positive even number;
 \item[(ii)] $\ker\varphi'$ is spanned by $\{mv'_i\mid{i}\geq{p-d-1},~m\in{B}\}$ if $d$ is odd; and by $mv'_{2p-1}$ if $d=0$; and by  $\{mv'_i\mid{i}\geq{2p-d-1},~m\in{B}\}$ if $d$ is a positive even number;
 \item[(iii)] $\ker\varphi$ is spanned by $\{mv_i\mid{i}\geq{p+d+1},m\in{B}\}$ if $d$ is odd, and by $\{mv_i\mid i=1,\cdots,2p-1,m\in{B}\}$ if $d=0$; spanned by $\{mv_i\mid{i}\geq{d+1},m\in{B}\}$ if $d$ is a positive even number.
 \item[(iv)] $\im\varphi$ is spanned by $\{mv'_i\mid{i}\geq p-d-1,m\in{B}\}$ if $d$ is odd, and by $mv'_{2p-1}$ if $d=0$, and by  $\{mv'_i\mid{i}\geq{2p-d-1},~m\in{B}\}$ if $d$ is a positive even number;
 \end{itemize}
 Then we have
\begin{align}\label{case 3 imagevarphi}
\ker\varphi'=\im\varphi=\begin{cases} \sum_{m\in B}\sum_{i=p-d-1}^{2p-1} mv'_i, & \mbox{ when }d \mbox{ odd},\cr
\sum_{m\in B}\sum_{i=2p-d-1}^{2p-1}  mv'_{i}, &\mbox{ when } d \mbox{ even}.
\end{cases}
\end{align}
Therefore, we have
\begin{align}\label{case 3 cokervarphi}
\coker\varphi&=\hat Z_A^{\hat r_\alpha}(\lambda-(2p-1)\alpha)\slash \im\varphi\cr
&=\hat Z_A^{\hat r_\alpha}(\lambda-(2p-1)\alpha)\slash \ker\varphi'\cr
&\cong \im\varphi'=\ker\varphi.
\end{align}
Similarly,  $\coker\varphi'\cong \ker\varphi'$. We complete the proof of (1).

For the proof of (2), we divide our argument into two cases. When $d$ is an odd number, $\ker\varphi$ is generated by $v_{p+d+1}$.  Thus, we can construct the following sequence of homomorphisms in $\mathcal C$
\begin{align*}
{\cdots}{\overset{\varphi_5}{\longrightarrow}}{\hat Z(\lambda-4p\alpha)}&{\overset{\varphi_4}{\longrightarrow}}{\hat Z(\lambda-(3p+d+1)\alpha)}\cr
&{\overset{\varphi_3}{\longrightarrow}}{\hat Z(\lambda-2p\alpha)}
{\overset{\varphi_2}{\longrightarrow}}{\hat Z(\lambda-(p+d+1)\alpha)}{\overset{\varphi_1}{\longrightarrow}}{\ker\varphi}\longrightarrow 0
\end{align*}
where $\varphi_i, i=1,2,3,\cdots$  are  defined  via $\varphi_1(1\otimes 1)=v_{p+d+1}$,  $\varphi_{2j}(1\otimes 1)= Y_\alpha^{p-d-1}\otimes 1$, and
 $\varphi_{2j+1}(1\otimes 1)=Y_\alpha^{p+d+1}\otimes 1$, $j=1,2, \cdots$ respectively.
It is easy to check that this sequence is exact.

 When $d$ is an even,  $\ker\varphi$ is generated by $v_{d+1}$.  We construct the following sequence of homomorphisms in $\mathcal C$
\begin{align*}
{\cdots}{\overset{\varphi_5}{\longrightarrow}}{\hat Z(\lambda-4p\alpha)}&{\overset{\varphi_4}{\longrightarrow}}{\hat Z(\lambda-(2p+d+1)\alpha)}\cr
&{\overset{\varphi_3}{\longrightarrow}}{\hat Z(\lambda-2p\alpha)}
{\overset{\varphi_2}{\longrightarrow}}{\hat Z(\lambda-(d+1)\alpha)}{\overset{\varphi_1}{\longrightarrow}}{\ker\varphi}\longrightarrow 0
\end{align*}
where $\varphi_1$ is  defined  via $\varphi_1(1\otimes 1)=v_{d+1}$, furthermore, $\varphi_i, i=2,3,\cdots$  are  defined  via   $\varphi_{2j}(1\otimes 1)= Y_\alpha^{2p-d-1}\otimes 1$, and
 $\varphi_{2j+1}(1\otimes 1)=Y_\alpha^{d+1}\otimes 1$, $j=1,2, \cdots$ respectively.
It is easy to check that this sequence is exact.  We complete proof.
\end{proof}

\begin{remark} \label{same Gro} The lemma implies that both (twisted) baby Verma modules arising from $\varphi$  define the same class in the Grothendieck group of $\mathcal{C}$.
\end{remark}
\subsection{Base change of exact sequences}
For given two commutative $\bbk$-algebras $A'$ and $A''$,  and a $\bbk$-algebra homomorphism  $f:A'\rightarrow{A''}$, we have  a functor from $\mathcal{C}_{A'}$ to $\mathcal{C}_{A''}$ with $f(M)=M\otimes_{A'}A''$. Especially, for (twisted) baby Verma modules we have $\hat Z_{A''}^{\hat w}(\lambda)=\hat Z_{A'}^{\hat w}(\lambda)\otimes_{A'}A''$.

\subsubsection{} From now on, we  will always set $A$ to be the localization of the polynomial ring $\bbk[t]$ in indeterminant $t$ at the maxiaml ideal generated by $t$ (unless other statement).  Then $A=\bbk+At$. There is a natural  homomorphism $\textsf{p}:A\rightarrow{A/{At}}\cong \bbk$. So we get a functor from $\mathcal{C}_A$ to $\mathcal{C}_\bbk$. Here $\mathcal{C}=\mathcal{C}_\bbk$ is actually the category $(U(\ggg),\frakt)\hmod$.

\begin{lemma} \label{base change diag}
Assume that  $A$ as above, and  $\pi(H_{\alpha})=ct$ for $c\in\bbk^\times$. The following statements hold.

(1) The  diagrams below of homomorphisms in $\mathcal{C}_A$ and $\mathcal{C}$ corresponding to different cases (Cases 1-3)  listed in the previous subsection are commutative:
\begin{itemize}
\item[in (Case 1),]
\begin{equation}       
  \begin{array}{ccccccc}   
  \hat Z_A(\lambda)& \overset{\varphi}{\longrightarrow} & \hat Z_A^{\hat r_{\alpha}}(\lambda-(p-1)\alpha)& \rightarrow & \coker{\varphi} & \rightarrow &0\\  
     \downarrow &  & \downarrow & &\downarrow\\
     \hat Z(\lambda)& \overset{\overline\varphi}{\longrightarrow} &\hat Z^{r_{\alpha}}(\lambda-(p-1)\alpha)& \rightarrow & \coker\overline{\varphi} & \rightarrow &0;\\
  \end{array}
\end{equation}

\item[in (Case 2),]
\begin{equation}       
  \begin{array}{ccccccc}   
  \hat Z_A(\lambda)&\overset{\varphi}{\longrightarrow} & \hat Z_A^{\hat r_{\alpha}}(\lambda-\alpha)& \rightarrow & \coker{\varphi} & \rightarrow &0\\  
     \downarrow &  & \downarrow & &\downarrow\\
     \hat Z(\lambda)&\overset{\overline\varphi}{\longrightarrow} &\hat Z^{s_{\alpha}}(\lambda-\alpha)& \rightarrow & \coker\overline{\varphi} & \rightarrow &0;
  \end{array}
\end{equation}

\item[in (Case 3),]
\begin{equation}       
  \begin{array}{ccccccc}   
  \hat Z_A(\lambda)&{\overset{\varphi}{\longrightarrow}} & \hat Z_A^{\hat r_{\alpha}}(\lambda-(2p-1)\alpha)& \rightarrow & \coker{\varphi} & \rightarrow &0\\  
     \downarrow &  & \downarrow & &\downarrow\\
     \hat Z(\lambda)&{\overset{\overline\varphi}{\longrightarrow}} &\hat Z^{r_{\alpha}}(\lambda-(2p-1)\alpha)& \rightarrow & \coker\overline{\varphi} & \rightarrow & 0
  \end{array}
\end{equation}
\end{itemize}
where $\overline\varphi$ has the same meaning as  the $\varphi$ from Lemma \ref{lemma case 3},  the homomorphisms in the first two columns arises from base change.

(2) The third vertical homomorphisms in each case are bijective.
\end{lemma}





\begin{proof} (I) For (Case 1), the same arguments can be done as in \cite{J2}. For the remaining cases, we easily shows the existance of the diagram because we can identify both $\hat Z(\lambda)$ and $\hat Z^{\hat r_\alpha}(\mu)$  with $\hat Z_A(\lambda)\otimes_A(A\slash At)$ and $\hat Z^{\hat r_\alpha}(\mu)\otimes_A(A\slash At)$ respectively.  So we only need to verify that the induced maps between cokernels are bijective for the remaining cases.

(II) We begin with (Case 2). By (\ref{case 2 comp}) we have
\begin{align}\label{varphi formula A case 2}
&\varphi(Amv)=\begin{cases}
Amv'_1,&{v=v_0},\cr
A(ct+d)mv_0',&{v=v_1};
\end{cases}\cr
&\varphi'(Amv')=\begin{cases}
Amv_1,&{v'=v'_0},\cr
A(ct+d)mv_0,&{v'=v'_1},
\end{cases}
\end{align}
where $m$, $v$ and $v'$ are the same as in the proof of Lemma \ref{lemma case 2}. When $\langle\lambda,\chi_\alpha\rangle\not\equiv 0(\mbox{mod }p)$, $ct+d$ is a unit in $A$. In this case both $\varphi$ and $\varphi'$ are isomorphisms. Hence $\coker\varphi=0=\coker\overline\varphi$. Assume $\langle\lambda,\chi_\alpha\rangle\equiv 0(\mbox{mod }p)$, then $\coker\varphi$ is spanned by all $(A\slash At) mv_0'$ with $m\in B$. Comparing with (\ref{image varphi case 2}), we have $\coker\varphi$ is mapped bijectively to $\coker\overline\varphi$.

(III) We now discuss  (Case 3). 
By the arguments in \S\ref{subsec case 3} we have for any $0\leq i \leq 2p-1$, $m\in B$
  \begin{align}\label{varphi formula A}
  \varphi(mv_i)=\begin{cases}
\prod_{j=1}^{\frac{i}{2}}(ct+d+2-2j)\prod_{j=1}^{\frac{i}{2}}(2(p-j))mv'_{2p-1-i},&\mbox{ when } i \mbox{~is~even}\\
\prod_{j=1}^{\frac{i+1}{2}}(ct+d+2-2j)\prod_{j=1}^{\frac{i-1}{2}}(2(p-j))mv'_{2p-1-i},&\mbox{ when } i \mbox{~is~odd},
\end{cases}
\end{align}
where $B$ is  a basis of $U_0(\mathfrak{m})$ as described in the proof of Lemma \ref{lemma case 3}, and $d=\langle\lambda,\chi_\alpha\rangle(\mbox{mod } p),~0\leq i\leq {2p-1}$. In the above formula,  $ct+d+2-2j$ is a unit  in $A$ if $d+2-2j\not\equiv 0 (\mbox{mod }p)$, otherwise it is a unit times $t$ in $A$. We continue our arguments into three cases according to the possibilities of  $d$.

(i) When  $d$ is an odd number, we have
\begin{align}\label{odd varphi A}
\varphi(Amv_i)=\begin{cases}
Amv'_{2p-1-i},&{0\leq{i}<p+d+1},\\
Atmv'_{2p-1-i},&{p+d+1\leq{i}\leq{2p-1}},
\end{cases}
\end{align}
therefore $\coker\varphi$ is spanned by  all $(A/At)mv'_i$ with $m\in{B}$ and $i\leq{p-d-2}$. Comparing with (\ref{case 3 imagevarphi}), we have $\coker\varphi\cong{\coker\overline{\varphi}}$.

(ii) When $d$ is an even number,  we have
\begin{align}\label{even varphi A}
\varphi(Amv_i)=\begin{cases}
Amv'_{2p-1},& 0\leq i< d+1,\\
Atmv'_{2p-1-i},& d+1\leq i\leq 2p-1,
\end{cases}
\end{align}
therefore $\coker\varphi$ is spanned by  all $(A/At)mv'_i$ with $m\in{B}$ and $i\leq 2p-d-2$. Comparing with (\ref{case 3 imagevarphi}), we have $\coker\varphi$ is mapped bijectively to ${\coker\overline{\varphi}}$.

Combining the above, we complete the proof.
\end{proof}

\begin{remark}\label{coker iso remark}
In the above lemma, we can regard $\mathcal{C}$ as a full subcategory of $\mathcal{C}_A$ (see the arguments in the beginning paragraph of the next section). Hence all maps of the diagrams are morphisms in $\mathcal{C}_A$. In particular,  $\coker(\varphi){\overset{\cong}{\longrightarrow}}\coker(\overline\varphi)$ in $\mathcal{C}_A$.
\end{remark}

\begin{remark} \label{c alpha} There exists a $\bbk$-algebra homomorphism $\pi:U^0\rightarrow \bbk[t]$ such that $\pi(H_\alpha)=c_\alpha t, \; \forall \alpha\in \Delta^+$ for some $c_\alpha\in \bbk^\times$. Actually, $\bbk$ is infinite while $\Delta^+$ is finite. So we can find a linear function $\pi_0:\hhh\rightarrow \bbk$ such that $\pi_0(H_\alpha)\ne 0$ for all $\alpha\in \Delta^+$. Extend $\pi_0$ to be a $\bbk$-algebra homomorphism $\pi$ from $U^0$ to $\bbk[t]$ by mapping $H_\alpha$ to $\pi_0(H_\alpha)t$. Then $\pi$ is desired.
\end{remark}

\subsubsection{} \label{inverse image} We look at some inverse images of $\varphi$. In Case 1, $\alpha\in \overline\Delta^+_0$. By the same arguments as reductive Lie algebra case (\cite[\S3.6]{J2}) we have for $m>0$, $\varphi^{-1}(t^m \hat Z_A^{\hat r_\alpha}(\lambda-(p-1)\alpha)\subset t^{m-1}\hat Z_A(\lambda)$, and we can say more that $\varphi^{-1}(t^m \hat Z_A^{\hat r_\alpha}(\lambda-(p-1)\alpha)=t^{m}\hat Z_A(\lambda)$ once $d=p-1$ for $d$ as in Lemma \ref{lemma case 1}.

 In Case 2, $\alpha\in \overline\Delta^+_1$. By (\ref{varphi formula A case 2}), we have for $m>0$, $\varphi^{-1}(t^m \hat Z_A^{\hat r_\alpha}(\lambda-\alpha)\subset t^{m-1}\hat Z_A(\lambda))$, and we can say more that $\varphi^{-1}(t^m \hat Z_A^{\hat r_\alpha}(\lambda-\alpha)=t^{m}\hat Z_A(\lambda)$ once $d\ne 0$.

 In Case 3, $\alpha\in  \Delta^+_1\backslash\overline\Delta^+_1$. By (\ref{odd varphi A}) and (\ref{even varphi A}) we have for $m>0$, $\varphi^{-1}(t^m \hat Z_A^{\hat r_\alpha}(\lambda-(2p-1)\alpha)\subset t^{m-1}\hat Z_A(\lambda)$.

\section{Jantzen filtration and sum formula}

In this section, we will construct a Jantzen filtration and formulate a sum formula for baby Verma modules on $\mathcal{C}$.  For this, we first make some preparations on $\mathcal{C_A}$.

\subsection{}
Throughout this section, we still  set $A$ to be the localization of $\bbk[t]$ at the maximal ideal generated by $t$. Let $Q(A)$ stand for the fractional field of $A$.  Note that there is a canonical $\bbk$-algebra embedding $\bbk[t]\hookrightarrow A$. We can choose a homomorphism $\pi:U^0\rightarrow A$ satisfying the condition in Remark \ref{c alpha}. This is to say,  $\pi(\mathfrak{h})\subset{At}$ satisfy  $\pi(H_\alpha)=c_\alpha t,\; c_\alpha\in \bbk^\times$.
Then  each objects in $\mathcal{C}$ can be regarded as an object in $\mathcal{C}_A$ via the surjection $A$ onto $A\slash tA\cong \bbk$. So we can regard $\mathcal{C}$ as a full subcategory of $\mathcal{C}_A$ consisting of all objects $M$ with $tM=0$.
It is obviously that the simple objects of $\mathcal{C}$ is also simple ones in $\mathcal{C}_A$. Conversely, if $M\in\text{obj}(\mathcal{C}_A)$ is simple in $\mathcal{C}_A$ then $tM=0$ or $tM=M$. However we cannot have $M=tM$   by Nakayama lemma because $M\ne 0$ and $M$ is finitely generated over $A$. So we have $tM=0$, and  $M$ is an object from $\mathcal{C}$, which means that $M$ is also simple in $\mathcal{C}$. 


Next we introduce a new category trsn($\mathcal{C}_A$) as in the reductive Lie algebra case (\textit{cf}. \cite{J2}), which is the subcategory of $\mathcal{C}_A$ and consist with torsion modules in $\mathcal{C}_A$.
The category trsn($\mathcal{C}_A$) can be described as the full subcategory of all objects of finite length in $\mathcal{C}_A$ (\textit{cf}. \cite[\S3.7]{J2}.
Denote by $\mbox{K}(\mbox{trsn}(\mathcal{C}_A))$ the Grothendieck group of the category trsn($\mathcal{C}_A$).

\subsection{} \label{general contru filt} Let $M$ and $M'$ be torsion free modules over $A$.  If $\varphi:M\rightarrow{M'}$ satisfies
  \begin{align}\label{Frac Cond}
  \varphi\otimes \id_{Q(A)}: M\otimes Q(A)\rightarrow M'\otimes Q(A) \mbox{ is an isomorphism}.
\end{align}
Then
$$\coker\varphi\in\text{objtrsn}(\mathcal{C}_A).$$
  Define $$\nu(\varphi):=[\coker\varphi]\in \mbox{K}(\mbox{trsn}(\mathcal{C}_A)).$$
  If $\psi:M'\rightarrow{M''}$ is another homomorphism of torsion free $A$-modules satisfying (\ref{Frac Cond}), then we have
  \begin{align}\label{sum for nu}
  \nu(\psi\circ\varphi)=\nu(\psi)+\nu(\varphi).
  \end{align}
The injectivity of $\psi$ shows that
 $$0\rightarrow{\coker\varphi}\rightarrow{\coker(\psi\circ\varphi)}\rightarrow{\coker\psi}\rightarrow0$$
  is an exact sequence.

We note that $\overline{\varphi}:\overline{M}\rightarrow\overline{M'}$ is homomorphism induced by $\varphi$, where $\overline{M}=M/{tM}$ and $\overline{M'}=M'/{tM'}$. We denote the natural functor from $\mathcal{C}_A$ to $\mathcal{C}$ by $\mathsf{p}$, i.e. $\mathsf{p}(M)=\overline{M}$.
We can define
$$M^i=\{m\in{M}\mid\varphi(m)\in{t^iM'}\}$$
and set $\overline{M}^i$ to be the image of $M^i$ in $\overline M$. Then $\overline M=\overline{M}^0=\overline M$.
It is obvious that $\overline{M}^1$ is the kernel of $\overline\varphi:\overline M\rightarrow \overline{M'}$. Hence
\begin{align}\label{quo filtr}
\overline{M}\slash{\overline{M}^1}\cong \im(\overline{\varphi}).
\end{align}
We turn to  other terms in the filtration $\{\overline{M}^i\}$. In view of  $\overline{M}\in$objtrsn($\mathcal{C}_A$),  there exists $n$ such that $t^n\overline{M}=0$. And $\overline{M}^i$ satisfy that $$\overline{M}=\overline{M}^0\supset{\overline{M}^1}\supset{\cdots}\supset{\overline{M}^n}=0,$$
so the filtration $\{\overline{M}^i\}$ of $\overline{M}=M/tM$ is of finite length.

We also have the next lemma
\begin{lemma} \label{comp nu lemma}
  $\sum_{i=1}^n[\overline{M}^i]=[\nu(\varphi)]$.
\end{lemma}

\begin{proof} It can be proved by the same arguments as in \cite[Lemma 3.8]{J2}.
\end{proof}


\subsection{Jantzen filtration for baby Verma modules}

\subsubsection{} \label{Cases call} By Theorem \ref{longest element thm}, there is a distinguished element $\hat w_0=\hat{r}_N\cdots{\hat{r}_2\hat{r}_1}$. Still denote $\hat\sigma_0=\id$, and $\hat\sigma_i=\hat{r}_i\cdots{\hat{r}_1}$ for $i=1,\cdots, N$ ($\hat\sigma_N=\hat w_0$).
For $\lambda\in Y$, we extend the notation $\lambda^{\hat r_\alpha}$ to  $\lambda^{\hat\sigma_i}$
\begin{align}\label{def lambda i1}
\lambda^{\hat \sigma_i}:=\lambda-(p-1)(1-\hat\sigma_i)(\rho_0)-(1-\sigma_i)(\rho_1)
\end{align}
where $\hat\sigma_i\rho_0:={1\over 2}\sum_{\alpha\in\hat\sigma_i(\Delta)^+_0}\alpha$ and $\sigma_i\rho_1:={1\over 2}\sum_{\alpha\in\hat\sigma_i(\Delta)^+_1}\alpha$
for  $\hat\sigma_i(\Delta)^+=\hat\sigma_i(\Delta)_0^+\cup\hat\sigma_i(\Delta)_1^+$  the positive root set associated with the simple root system $\hat\sigma_i(\Pi)$. Then we have
\begin{align}\label{def lambda i2}
(\lambda-\mu)^{\hat\sigma_i}=\lambda^{\hat\sigma_i}-\mu \mbox{ for any }\mu\in Y.
\end{align}

We will often denote $\lambda^{\hat \sigma_i}$ by $\lambda_i$ for simplicity.
Next we can describe the inductive relation between those $\lambda_{i-1}$ and $\lambda_{i}$:
\begin{align*}\lambda_{i}=\begin{cases}
\lambda_{i-1}-(p-1)\theta_{i}, &{\mbox{ if }} \theta_{i}\in \overline\Delta_0^+; \cr
\lambda_{i-1}-\theta_{i};  &{\mbox{ if }} \theta_{i}\in \overline\Delta_1^+;\cr
\lambda_{i-1}-(2p-1)\theta_{i}, &{\mbox{ if }} \theta_{i}\in \Delta_1^+\backslash\overline\Delta_1^+.
 \end{cases}
\end{align*}
The inductive calculation shows that  $\lambda_N=\lambda-2(p-1)\rho_0-2\rho_1$. In the sequent arguments,
{\sl {we still speak of (Case 1) for $\theta_i\in \overline \Delta_0^+$ and (Case 2) for $\theta_i\in\overline\Delta_1^+$) and (Case 3) for $\theta_{i}\in \Delta_1^+\backslash \overline \Delta_1^+$.}}

By the same construction as in \S\ref{baby Verma hom}, of $\varphi:\hat Z_A(\lambda)\rightarrow\hat Z_A^{\hat r_\alpha}(\lambda_1)$, we can construct a series of homomorphisms $$\Xi_{i-1}:\hat Z_A^{\hat\sigma_{i-1}}(\lambda_{i-1})\rightarrow{\hat Z_A^{\hat\sigma_i}(\lambda_i)}, i=1,2,\cdots,N$$
which share  the same properties as $\varphi$ presented in \S\ref{baby Verma hom}, and therefore we have a sequence of homomorphisms
$$\hat Z_A(\lambda){\overset{\Xi_0}{\longrightarrow}}{\hat Z_A^{\hat\sigma_1}(\lambda_1)}{\overset{\Xi_1}{\longrightarrow}}{\hat Z_A^{\hat\sigma_2}(\lambda_2)}{\overset{\Xi_2}{\longrightarrow}}\cdots\cdots{\overset{\Xi_{N-2}}{\longrightarrow}}\hat Z_A^{\hat \sigma_{N-1}}(\lambda_{N-1}){\overset{\Xi_{N-1}}{\longrightarrow}}{\hat Z_A^{\hat \sigma_N}(\lambda_N)}.$$
Denote the composition of these sequence of homomorphisms by  $\Xi$, i.e.
 $$\Xi=\Xi_{N-1}\circ\cdots\circ\Xi_{1}\circ\Xi_0:\hat Z_A(\lambda)\longrightarrow {\hat Z_A^{\hat w_0}(\lambda_N)}.$$

\subsubsection{}\label{nonzero Xi} By the construction of all $\Xi_i$'s, we can show that $\mathsf{p}(\Xi)$ is a non-zero homomorphism from $\hat Z(\lambda)$ to $\hat Z^{\hat w_0}(\lambda_N)$. Actually, by a straightforward computation
\begin{align} \label{image Xi v0}
 \Xi(1\otimes 1)=X_1^{r_1}\cdots X_{N}^{r_N}\otimes 1 \in \hat Z_A^{\hat w_0}(\lambda_N)
\end{align}
 with  $X_i=X_{\theta_i}$ and $r_i=p-1$ for $\theta_i$ in (Case 1), $1$ for $\theta_i$ in (Case 2), and $2p-1$ for $\theta_i$ in (Case 3) respectively. Hence $\mathsf{p}(\Xi)$ contains the nonzero vector $X_1^{r_1}\cdots X_{N}^{r_N}\otimes 1 \in \hat Z^{\hat w_0}(\lambda_N)$.

There is another interesting observation which will be useful for the sequent arguments. Set
$$v=Y_N^{r_N}\cdots Y_{1}^{r_1}\otimes 1\in \hat Z_A(\lambda)$$
 with  $Y_i=Y_{\theta_i}$ and $r_i$ as above for $i=1,\cdots,N$.
Then by (\ref{varphi formula A}) we have
$$\Xi_i(Y_{i}^{r_i}\otimes 1)=b_i t\otimes 1$$
  for some unit $b_i$ in $A$ whenever $\theta_i$ is in (Case 3). As to
the other two cases, by \cite[\S3.6]{J2} and (\ref{case 2 comp}) we have
\begin{align*}
&\Xi_i(Y_i^{r_i}\otimes 1)\cr
=&\begin{cases}
 tb_i\otimes 1 \mbox{ if }m_j<p \mbox{ when }\theta_{i} \mbox{ in (Case 1)}, \mbox{ or if }m_i=p\mbox{ when }\theta_{i} \mbox{ in (Case 2)}; \cr
b_i\otimes 1 \mbox{ if }m_j=p \mbox{ when }\theta_{i} \mbox{ in (Case 1)}, \mbox{ or if }m_i\ne p\mbox{ when }\theta_{i} \mbox{ in (Case 2)}
\end{cases}
\end{align*}
where $m_i\in \{1,\cdots,p\}$ with $m_i\equiv \langle\lambda_{i-1},\chi_{\theta_i}\rangle(\mbox{mod }p)$, and  $b_i$ is some unit in $A$. Set
\begin{align*} 
N_\lambda:=\#\{\theta_i\in \overline \Phi_0^+\mid m_i<p-1\}
+\#\{\theta_i\in\overline\Phi^+_1\mid m_i=p)\}
+\# \Phi_1^+\backslash\overline\Phi_1^+.
\end{align*}
By an inductive calculation, we finally have that there is a unit $b\in A$ such that
\begin{align}\label{image Xi}
\Xi(v)=t^{N_\lambda}b\otimes 1.
\end{align}

\begin{remark} At the first glance, the expression of $N_\lambda$ is dependent on the inductive steps. However, from the proof of Theorem \ref{Main Thm} we will read off  a concise expression independent of any inductive steps, which is as follows:
 \begin{align} \label{N lambda}
N_\lambda=\#\{\alpha\in \overline \Delta_0^+\mid m_\alpha<p\}
+\#\{\gamma\in\overline\Delta^+_1\mid m_\gamma=p)\}
+\# \Delta_1^+\backslash\overline\Delta_1^+,
\end{align}
where $m_\theta\in \{1,\cdots,p\}$ with $m_\theta\equiv \langle\lambda+\rho,\chi_\theta\rangle(\mbox{mod }p)$ for $\theta\in \Delta^+$.

\end{remark}

\subsubsection{} Now we investigate more about  both modules involved  from $\Xi$ in the category $\mathcal{C}$. Before that, consider linear dual module  $M^*$ for an superspace  $M=M_\bo+M_\bz$ in $\mathcal{C}_A$.   The module structure on the superspace $M^*=(M_\bo)^*+(M_\bz)^*$  is given via $X\cdot\phi=(-1)^{|X||\phi|+1}\phi\circ X$ for  homogeneous elements $X\in \ggg_{|X|}$ and $\varphi\in M^*_{|\varphi|}$. It is readily shown that $M^*$ is still in the category $\mathcal{C}$. Actually, for $M=\sum_{\mu\in Y}M_\lambda$, we have  $M^*=\sum_{\mu\in Y}(M^*)_\mu$ with $(M^*)_{\mu}=(M_{-\mu})^*$. So it is easy to check $$\hat Z_{0}(\lambda)^*\cong{\hat Z_{0}(-\lambda+(p-1)2\rho_0+2\rho_1)}$$ by comparing the highest weight of $\hat Z_{0}(\lambda)^*$ and $\hat Z_{0}(-\lambda+(p-1)2\rho_0+2\rho_1)$.

\begin{lemma} \label{soc twistest}
The socle of $\hat Z_{0}^{\hat w_0}(\lambda_N)$ is isomorphic to $\hat L(\lambda)$.
\end{lemma}

\begin{proof} By the arguments in \S\ref{tau}, there is a standard involution $\tau_0\in \Aut(\ggg)$ (and a standard involution $\tau\in \Aut(G)$). This $\tau_0$ or $\tau$ induces a self-equivalence functor  on $\mathcal{C}$, sending  $M$ to $\ttau{M}$ where $\ttau{M}$ is $M$ itself as a superspace, with action as
$X\cdot m=\tau^{-1}(X)m$ for $X\in \ggg$ and $m\in M$. It is easily known that for $M=\sum_{\lambda\in Y}M_\lambda\in \mathcal{C}$
$$\ttau{M}=\sum_{\lambda\in Y}(\ttau{M})_\lambda \mbox{ with }(\ttau{M})_\lambda=M_{-\lambda}.$$
Especially, for $\ttau{\hat Z(\lambda)}\cong\hat U_0(\ggg)\otimes_{U_0(\hhh+\nnn^-)}\bbk_{-\lambda}=\hat Z^{\hat w_0}(-\lambda)$. Thus we have
\begin{align}\label{tau duality}
\ttau(\hat Z_{0}(\lambda)^*)\cong \hat Z_{0}^{\hat w_0}(\lambda_N).
\end{align}
Note that linear duality changes a head into a socle. On the other hand, $\ttau(\hat L_{0}(\lambda)^*)$ is still a simple object in $\mathcal{C}$ with highest weight $\lambda$, thereby must be isomorphic to $\hat L(\lambda)$.  So (\ref{tau duality}) implies that $\mbox{Soc}(\hat Z_{0}^{\hat w_0}(\lambda_N)\cong\hat L(\lambda)$. We complete the proof.
  \end{proof}


\subsubsection{} We return to the homomorphism in $\mathcal{C}_A$
$$\Xi: \hat Z_A(\lambda)\longrightarrow\hat Z_{A}^{\hat w_0}(\lambda_N).$$
 Remark \ref{same Gro} is suitable to each $\Xi_i$. So we have an important consequence that all $\hat Z^{\hat \sigma_i}(\lambda_i)$ define the same class in the Grothendieck group of $\mathcal{C}$ as $\hat Z(\lambda)$, i.e.
 \begin{align} \label{Gro Class Eq}
 [\hat Z^{\hat \sigma_i}(\lambda_i]=[\hat Z(\lambda)], i=0,1,\cdots,N.
 \end{align}
 which will be used later.
By the definition, $\hat Z_A(\lambda)$ and $\hat Z_{A}^{\hat w_0}(\lambda_N)$ are torsion free over $A$. Furthermore, Formula (\ref{varphi formula A case 2}) and Formulas (\ref{odd varphi A})-(\ref{even varphi A}) along with \cite[\S3.6(2)]{J2}  are suitable to all $\Xi_{i}:\hat Z_A^{\hat \sigma_{i}}(\lambda_{i})\rightarrow \hat Z_A^{\hat \sigma_{i+1}}(\lambda_{i+1}), \; i=0,1,\cdots,N-1$. This implies that if we work over the fractional field $Q(A)$ by base change, then for $i$
$$\Xi_{i}\otimes{\id_{Q(A)}}:~\hat Z_A^{\hat\sigma_{i}}(\lambda_{i})\otimes{Q(A)}\rightarrow{\hat Z_{A}^{\hat \sigma_{i+1}}(\lambda_{i+1})\otimes{Q(A)}}$$
 become isomorphisms. Hence $\Xi\otimes \id_{Q(A)}$ becomes an isomorphism.  So $\Xi$ satisfies the condition (\ref{Frac Cond}). Hence we can construct through $\Xi$ the filtration as introduced in \S\ref{general contru filt}.
 By application of general filtration construction in \S\ref{general contru filt} to the baby Verma module case, we can get a filtration $\{\hat Z_A(\lambda)^i\}$  and then $\{\hat Z(\lambda)^i\}$, where $$\hat Z_A(\lambda)^i=\{v\in{\hat Z_A(\lambda)}\mid\Xi(v)\in{t^i\hat Z_{0}^{\hat w_0}(\lambda_N)}\}$$ and $$\hat Z(\lambda)^i\cong \overline{\hat Z_A(\lambda)^i}.$$


\subsection{Jantzen sum formula}
\begin{theorem}\label{Main Thm}
For the above  filtration  $\{\hat Z(\lambda)^i\}$ of $\hat Z(\lambda)$, the following statements hold.

 (1) There is a  sum formula in the Grothendieck group of $(U_0(\ggg),\frakt)\hmod$
\begin{align*}
\sum_{i>0}[\hat Z(\lambda)^i]=&\sum_{\alpha\in \overline{\Delta}^+_0}(\sum_{i\geq0}[\hat Z(\lambda-(ip+m_\alpha)\alpha)]-\sum_{i>0}[\hat Z(\lambda-ip\alpha)])\cr
+&\sum_{\gamma\in{\overline{\Delta}^+_1}, m_\gamma=p}(\sum_{i\geq 0}[\hat Z(\lambda-(2i+1)\gamma)]-\sum_{i>0}[\hat Z(\lambda-2i\gamma)])
\cr
+&\sum_{\beta\in\Delta^+_1\backslash{\overline{\Delta}_1}}
(\sum_{i\geq 0}[\hat Z(\lambda-((2i+\delta_{1,(-1)^{m_\beta-1}})p+m_\beta)\beta)]-\sum_{i>0}[\hat Z(\lambda-(2ip)\beta)])
\end{align*}
where $\delta_{i,j}:=1$ if $i=j$, $\delta_{i,j}:=0$ if $i\ne j$; and $m_\theta\in\{1,\cdots,p-1,p\}$
with
$m_\theta\equiv\langle\lambda+\rho,\chi_{\theta}\rangle(\mbox{mod }p)$ for $\theta\in \Delta^+$.

(2) $\hat Z(\lambda)^i=0$ if and only if $i> N_\lambda$  where $N_\lambda$ is as in (\ref{N lambda}).


(3)   $\hat Z(\lambda)/{\hat Z(\lambda)^1}\cong{\hat L(\lambda)}$.

\end{theorem}

\begin{proof} (1) By Lemma \ref{comp nu lemma} and the formula \ref{sum for nu},  we know that
\begin{align}\label{sum coker formula}
\sum_{i>0}[\hat Z(\lambda)^i]=\sum_{j=0}^{N-1}[\coker\Xi_j].
\end{align}
Thanks to Lemma \ref{base change diag} along with Remark \ref{coker iso remark}, we know $[\coker(\Xi_j)]=[\coker(\overline\Xi_j)]$ for all $j=0,\cdots, N-1$ where $\overline\Xi_j:\hat Z^{\hat \sigma_j}(\lambda_j)\rightarrow \hat Z^{\hat \sigma_{j+1}}(\lambda_{j+1})$, a homomorphism in $\mathcal{C}$ which is gotten from $\Xi_i$ by base change arising from  the projection  $A\twoheadrightarrow\bbk\cong A\slash tA$.

 For the further arguments, we have to  divide them into three  cases,  according to the situations for $\theta_{j+1}$.

(i) Suppose $\theta_{j+1}\in \Delta^+_1\backslash \overline\Delta_1^+$.
Lemma \ref{lemma case 3} is suitable to $\Xi_j$. Replace $d$ in Lemma \ref{lemma case 3} by $m_j$, i.e. $m_j\in \{0,1,\cdots,p-1\}$ with $m_j\equiv\langle \lambda_{j},\chi_{\theta_{j+1}}\rangle (\mbox{mod }p)$. So we have
 $$[\coker\overline\Xi_j]=\sum_{i\geq 0}[\hat Z^{\sigma_{j}}(\lambda_{j}-((2i+\delta_{1,(-1)^{m_j}})p+m_j+1)\theta_{j+1})]-\sum_{i>0}[\hat Z^{\sigma_{j}}(\lambda_{j}-(2ip\theta_{j+1}))].$$
According to (\ref{Gro Class Eq}), we know
$$[\hat Z^{\sigma_{j}}(\lambda_{j})]=[\hat Z(\lambda)].$$
Replacing $\lambda$ by $\lambda-2ip\theta_{j+1}$ or
$\lambda-((2i+\delta_{1,(-1)^{m_j}})p+m_j+1)\theta_{j+1}$, then by the definition of $\lambda_{j}=\lambda^{\hat\sigma_j}$ (\ref{def lambda i1}) and its property  (\ref{def lambda i2}) we have
$$[\hat Z^{\sigma_{j}}(\lambda_{j}-2ip\theta_{j+1})]=[\hat Z(\lambda-2ip\theta_{j+1})]$$
and
$$[\hat Z^{\sigma_{j}}(\lambda_{j}-((2i+\delta_{1,(-1)^{m_j}})p+m_j+1)\theta_{j+1})]=[\hat Z(\lambda-((2i+\delta_{1,(-1)^{m_j}})p+m_j+1)\theta_{j+1})].$$
 So we have
$$[\coker\Xi_j]=\sum_{i\geq 0}[\hat Z(\lambda-((2i+\delta_{1,(-1)^d})p+m_j+1)\theta_{j+1})]-\sum_{i>0}[\hat Z(\lambda-(2ip)\theta_{j+1})].$$

(ii) Suppose $\theta_{j+1}\in \overline\Delta_1^+$. By the same arguments as above, with application of Lemma \ref{lemma case 2} we get for $m_j=0$
$$[\coker\Xi_j]=\sum_{i\geq0}[\hat Z(\lambda-(2i+1)\theta_{j+1})]-\sum_{i>0}[\hat Z(\lambda-2i\theta_{j+1})].$$
As to the case when $m_j\ne 0$, $\coker\Xi_j=\coker\overline\Xi_j=0$ by Lemma \ref{lemma case 2}.

(iii) Suppose $\theta_{j+1}\in \overline\Delta_0^+$. By the same arguments as the above (i), with application of Lemma \ref{lemma case 1} we get for the case when $m_j<p-1$
$$[\coker\Xi_j]=\sum_{i\geq0}[\hat Z(\lambda-(ip+m_j+1))\theta_{j+1}]-\sum_{i>0}[\hat Z(\lambda-ip\theta_{j+1})]$$
where   $m_j$'s are in the same sense as the $d$ from  Lemma \ref{lemma case 1}, i.e. $m_j\in \{0,1,\cdots,p-1\}$ with $m_j=\langle\lambda_{j},\chi_{\theta_{j+1}}\rangle (\mbox{mod }p)$. The above formula is also  suitable to the case  when $m_j=p-1$ because in the  case,  $\coker\Xi_j=0$ by Lemma \ref{lemma case 1}, and the sum on the right-hand side becomes $0$.

Summing up, Formula (\ref{sum coker formula}) becomes
\begin{align*}
\sum_{i>0}[\hat Z(\lambda)^i]=&\sum_{\theta_{j+1}\in \overline{\Delta}^+_0}(\sum_{i\geq0}[\hat Z(\lambda-(ip+m_j+1)\theta_{j+1})]-\sum_{i>0}[\hat Z(\lambda-ip\theta_{j+1})])\cr
+&\sum_{\theta_{j+1}\in{\overline{\Delta}^+_1},m_j=0}(\sum_{i\geq 0}[\hat Z(\lambda-(2i+1)\theta_{j+1})]-\sum_{i>0}[\hat Z(\lambda-2i\theta_{j+1})])
\cr
+&\sum_{\theta_{j+1}\in\Delta^+_1\backslash{\overline{\Delta}_1}}(\sum_{i\geq 0}[\hat Z(\lambda-((2i+\delta_{1,(-1)^{m_j}})p+m_j+1)\theta_{j+1})]\cr
&\quad \quad\qquad-\sum_{i>0}[\hat Z(\lambda-(2ip)\theta_{j+1})]).
\end{align*}
Set $m_j\in \{0,1,\cdots,p-1\}$ with $m_j\equiv\langle\lambda_{j},\chi_{\theta_{j+1}}\rangle (\mbox{mod }p)$. In the first and third summands on the right hand side of the above formula,
\begin{align*}
m_j+1&\equiv\langle\lambda_{j},\chi_{\theta_{j+1}}\rangle+1 (\mbox{mod }p)\cr
&=\langle\lambda-(p-1)(\rho_0-\sigma_j\rho_0)-(\rho_1-\sigma_j\rho_1),\chi_{\theta_{j+1}}\rangle+1(\mbox{mod }p)\cr
&=\langle\lambda+(\rho_0-\rho_1), \chi_{\theta_{j+1}}\rangle-\langle\sigma_j(\rho_0)-\sigma_j(\rho_1), \chi_{\theta_{j+1}}\rangle+1 (\mbox{mod }p).
\end{align*}
Note that by the definition, $\sigma_j(\rho_0)-\sigma_j(\rho_1)$ is the Weyl vector with respect to $\sigma_j(\Delta)^+$ the positive root set corresponding to the fundamental system $\sigma_j(\Pi)$, and $\theta_{j+1}$ is a simple root of $\sigma_j(\Pi)$ (Theorem \ref{longest element thm}(1)), and non-isotropic. By \cite[Proposition 1.33]{CW} and \cite[\S3.1]{FG}, we have $\langle\sigma_j(\rho_0)-\sigma_j(\rho_1), \chi_{\theta_{j+1}}\rangle={1\over 2}\langle\theta_{j+1}, \chi_{\theta_{j+1}}\rangle=1$. Hence, modulo $p$ we can replace $m_j+1$ by
$\langle\lambda+\rho,\chi_{\theta_{j+1}}\rangle$ with $\theta_{j+1}$ running over $\overline\Delta_0^+\cup (\Delta_1^+\backslash\overline\Delta_1^+)$.

For $\theta_{i+1}\in \overline\Delta_1^+$, by the same arguments as above we have $\langle\lambda_j,\chi_{\theta_{j+1}}\rangle\equiv \langle\lambda+\rho,\chi_{\theta_{j+1}}\rangle(\mbox{mod }p)$ because  $(\theta_{j+1}, \theta_{j+1})=0$. So modulo $p$ we can regard $m_j=\langle\lambda+\rho,\chi_{\theta_{j+1}}\rangle$.

Thus, we finally get the desired sum formula as expressed in the theorem.
We prove the statement (1).

(2) Now we apply the arguments on some inverse images of $\varphi$ in \S\ref{inverse image} to all $\Xi_{j}$. we can describe an inverse image of
$\varphi$ for $m>0$:
$$\Xi_j^{-1}(t^m\hat Z_A^{\hat\sigma_{j+1}}(\lambda_{j+1}))\subset t^{m-1}\hat Z_A^{\hat\sigma_{j}}(\lambda_{j})$$
 when $\theta_{j+1}\in \Delta_1^+\backslash \overline \Delta_1^+$, i.e. in (Case 3) as called in \S\ref{Cases call}. As to the other two cases  ((Case 1) for $\overline \Delta_0^+$ and (Case 2) for $\overline\Delta_1^+$), we have
\begin{align*}
&\Xi_j^{-1}(t^m\hat Z_A^{\hat\sigma_{j+1}}(\lambda_{j+1}))\cr
=&\begin{cases}
\subset t^{m-1}\hat Z_A^{\hat\sigma_{j}}(\lambda_{j}) \mbox{ if }m_j<p-1 \mbox{ when }\theta_{j+1} \mbox{ in (Case 1)}, \mbox{ or if }m_j=0\mbox{ when }\theta_{j+1} \mbox{ in (Case 2)}; \cr
=t^{m}\hat Z_A^{\hat\sigma_{j}}(\lambda_{i})  \mbox{ if }m_j=p-1 \mbox{ when }\theta_{j+1} \mbox{ in (Case 1)}, \mbox{ or if }m_j\ne 0\mbox{ when }\theta_{j+1} \mbox{ in (Case 2)}.
\end{cases}
\end{align*}
 So we have
$$\Xi^{-1}(t^{N_\lambda+1} \hat Z_A^{\hat w_0}(\lambda_N))\subset t\hat Z_A(\lambda)$$
which implies that  $\hat Z(\lambda)^{N_\lambda+1}=0$. On the other hand, by (\ref{image Xi}) we have known  that there is a distinguished vector $v=Y_1^{r_1}\cdots Y_N^{r_N}\otimes 1\in \hat Z_A(\lambda)^{N\lambda}$. We can think that  the nonzero vector $v$ is contained in $Z(\lambda)^{N_\lambda}$. The statement (2) is proved.

(3) From  (\ref{quo filtr}), it follows that  $$\hat Z(\lambda)\slash {\hat Z(\lambda)}^1\cong\im(\overline\Xi).$$
Note that $\hat Z(\lambda)$ has a unique simple quotient $\hat L(\lambda)$, and $\overline\Xi$ is a nontrivial homomorphism (see the arguments around (\ref{image Xi v0})). Hence $\im(\overline\Xi)$ has a simple head isomorphic $\hat L(\lambda)$.  From  $[\hat Z(\lambda):\hat L(\lambda)]=1$, it follows that $[\im(\overline\Xi):\hat L(\lambda)]=1$. On the other hand, by Lemma \ref{soc twistest} we know that $\hat Z^{\hat w_0}(\lambda_N)$ has a simple socle isomorphic to $\hat L(\lambda)$, thereby the submodule $\im(\overline \Xi)$  naturally has a simple socle isomorphic to $\hat L(\lambda)$.
Hence  $\im(\overline\Xi)\cong \hat L(\lambda)$.
We complete the proof of the statement (3).
\end{proof}

\section{Strong linkage principle}

In the concluding section, we will use Theorem \ref{Main Thm} to get a strong  linkage principle in $\mathcal{C}$.

\subsection{} Set $\Delta_c^+= \overline\Delta_0^+\cup (\Delta^+_1\backslash \overline\Delta_1^+)$. Then the super reflection $\hat r_\alpha$ is a real reflection, simply written as $r_\alpha$. Denote by $r_{\alpha,n}$ for $n\in \bbz$ the affine reflection given by  given by $ r_{\alpha,n}(\lambda)=r_\alpha(\lambda)+n\alpha$ for any $\lambda\in Y$.  Denote by $W_p$ the affine Weyl group  generated by all $r_{\alpha,np}$ with $n\in\bbz$ and $\alpha\in \Delta_c^+$. Define $w\cdot\lambda=w(\lambda+\rho)-\rho$ for $w\in W_p$. 

\begin{definition} Let $\lambda,\mu\in Y$.
\begin{itemize}
\item[(1)] Call $\mu\uparrow \lambda$ if either $\mu=\lambda$, or there are affine reflections $r_1,\cdots,r_s\in W_p$ such that
$$\lambda\geq r_1\cdot\lambda\geq r_2r_1\cdot\lambda\geq \cdots\geq (r_s\cdots r_2r_1)\cdot\lambda=\mu.$$
\item[(2)] Call $\mu\upuparrows\lambda$ if $\mu\leq\lambda$ and there exist  $\mu',\lambda'\in Y$ 
    such that  $\mu'\uparrow \lambda'$  with $\mu=\mu'$ modulo $\bbz\overline\Delta_1^+$ and  $\lambda=\lambda'$  modulo $\bbz\overline\Delta_1^+$. 
       Here $\bbz\overline\Delta_1^+$ means the free abelian group spanned by $\overline\Delta_1^+$.
\end{itemize}
\end{definition}

Note that the relation $\uparrow$ is transitive. 
So, it is easy to see that the relation $\upuparrows$ is transitive, this is to say, if $\mu\upuparrows\lambda$ and $\kappa\upuparrows \mu$ then $\kappa\upuparrows \lambda$.

\begin{theorem} \label{linkage thm} Let $\lambda,\mu\in Y$. If $[\hat Z(\lambda),\hat L(\mu)]\ne 0$, then $\mu\upuparrows \lambda$.
\end{theorem}

We will leave the proof of the theorem in \S\ref{proof sub}.

\begin{remark} One can  compare the strong linkage principle formulated here  with some related results from \cite{Marko, MarkoZ2}).
\end{remark}

\subsection{}\label{linkage rel} For $\lambda\in Y$ and $\theta\in \Delta_c^+\cup \overline\Delta^+_1$, keep the notation  $m_\theta$ as in Theorem \ref{Main Thm}. We can write $\langle\lambda+\rho,\chi_\theta\rangle=np+m_\theta$ for some $n\in \bbz$, and make the following convention:
\begin{itemize}
\item{}
When $\theta$ is in (Case 1), set for $i\geq 0$
$$\lambda_{2i}=\lambda-ip\theta,\;\; \lambda_{2i}=\lambda-(ip+m_\theta)\theta;$$
\item{}
When $\theta$ is in (Case 2), set for $i\geq 0$
$$\lambda_{2i}=\lambda-2i\theta,\;\; \lambda_{2i+1}=\lambda-(2i+1)\theta;$$
\item{}
When $\theta$ is in (Case 3), set for $i\geq 0$
$$\lambda_{2i}= \lambda-2ip\theta,\;\; \lambda_{2i+1}=\lambda-((2i+\delta_{1,(-1)^{m_\theta-1}})p+m_\theta)\theta$$
\end{itemize}
Then the summand corresponding to $\theta$ on the right hand side of the sum formula in Theorem \ref{Main Thm}(1) becomes
$$\sum_{i\geq 0}[\hat Z(\lambda_{2i+1})]-\sum_{i>0}[\hat Z(\lambda_{2i})].$$

\begin{lemma} \label{linkage lemma} For  given $\lambda\in Y$ and $\theta\in \Delta^+_c\cup\overline\Delta^+_1$, $\lambda_{i+1}\upuparrows\lambda_i$ for all $i\geq 0$.
\end{lemma}
\begin{proof} We prove this lemma by case-by-case arguments.

(i) When $\theta$ is in (Case 1), from $\lambda\geq\lambda_1=r_{\theta,np}\cdot\lambda$ it follows that $\lambda_1\uparrow \lambda_0=\lambda$. Furthermore, we get for $i>0$, $\lambda_{2i-1}\geq \lambda_{2i}=r_{\alpha, (n-(2i-1))p}\cdot\lambda_{2i-1}$ and $\lambda_{2i}\geq \lambda_{2i+1}=r_{\theta,(n-2i)p}\cdot\lambda_{2i}$. In this case, the claim is true.

(ii) When $\theta$ is in (Case 2), the claim is obviously true.

(iii) When $\theta$ is in (Case 3), set $\epsilon= \delta_{1,(-1)^{m_\theta-1}}$. From $\lambda\geq\lambda_1=r_{\theta,(n-(2+\epsilon))p}\cdot\lambda$ it follows that $\lambda_1\uparrow \lambda_0=\lambda$. Furthermore, we get for $i>0$, $\lambda_{2i-1}\geq \lambda_{2i}=r_{\theta, (n-(4i-2)-\epsilon))p}\cdot\lambda_{2i-1}$ and $\lambda_{2i}\geq \lambda_{2i+1}=r_{\theta,(n-4i-\epsilon)p}\cdot\lambda_{2i}$. In this case, the claim is true.

Summing up, we complete the proof.
\end{proof}

\subsection{}\label{proof sub} Now we prove Theorem \ref{linkage thm}. If $[\hat Z(\lambda),\hat L(\mu)]\ne 0$, then $\mu\leq\lambda$. Hence we can prove the theorem by induction on the height of $\lambda-\mu$ which is equal to $\sum_{\alpha\in \Pi}c_\alpha$ for $\lambda-\mu=\sum_{\alpha\in \Pi}c_\alpha\alpha$, $c_\alpha\in \bbz_{\geq 0}$.

When $\mu=\lambda$,  the claim is obvious. We suppose $\mu<\lambda$. By Theorem \ref{Main Thm}(3), $\hat L(\mu)$ has to be a  composition factor of the first term $\hat Z(\lambda)^1$ of the Jantzen filtration. Owing to Theorem \ref{Main Thm}(1), there exists $\theta\in \Delta^+_c\cup\overline\Delta^+_1$ such that $\hat L(\mu)$ is a composition factor of $\hat Z(\lambda_{2i+1})$ for some $i\geq 0$, using the notation from \S\ref{linkage rel}. By induction we get $\mu\upuparrows \lambda_{2i+1}$. Combining with Lemma \ref{linkage lemma} and the transitive property of the relation $\upuparrows$, we complete the proof of the  theorem.

\subsection{} By the same arguments  as above (with checking more in (Case 2)), we can give another version of strong linkage principle, which can be regarded as an optimal version.
\begin{prop} Let $\lambda,\mu\in Y$. If $[\hat Z(\lambda),\hat L(\mu)]\ne 0$, then there exist $s\in \bbn$ and a series of weights $\lambda_i$, $i=0,1,\cdots,s$ satisfying
\begin{align}\label{second linked}
\lambda&=\lambda_0\geq \lambda_1\geq \cdots \geq \lambda_{s-1}\geq \lambda_s=\mu, \mbox{ such that } \cr
&\mbox{ either } \lambda_i=r_{\alpha_i,n_ip}\cdot\lambda_{i-1} \mbox{ for some }\alpha_i\in \Delta_c^+, n_i\in \bbz\cr
&\mbox{ or } \lambda_i=\lambda_{i-1} -\gamma_i \mbox{ for some }\gamma_i\in \overline\Delta_1^+\mbox{ with } p|\langle\lambda_{i-1}+\rho,\chi_{\gamma_i}\rangle.
\end{align}
\end{prop}

We further reformulate the statement in the proposition  concisely by defining $\mu\Uparrow\lambda$ if either $\lambda=\mu$ or they satisfy (\ref{second linked})

\begin{theorem}\label{stronglinkage2} Let $\lambda,\mu\in Y$. If $[\hat Z(\lambda),\hat L(\mu)]\ne 0$, then $\mu\Uparrow\lambda$.
\end{theorem}

\begin{remark} One can expect to have Jantzen filtration and sum formula for Weyl modules of algebraic supergroups (\textit{cf}. \cite{LS}).
\end{remark}

\subsection*{Acknowledgement} The second-named author thanks Shun-Jen Cheng and Changjie Cheng very much for helpful discussion.

\end{document}